\documentclass[11pt]{amsart}
\usepackage{latexsym}
\usepackage{multicol}
\usepackage{verbatim}
\usepackage{hyperref}
\usepackage{amsmath, amscd}
\usepackage{color}
\usepackage[latin1]{inputenc}
\usepackage{enumerate}
\usepackage{amsfonts,amsmath,amsthm,amssymb,amscd}
\usepackage{latexsym}
\usepackage{mathrsfs}

\usepackage{mathdots} 

\advance\textwidth by 1.2in \advance\oddsidemargin by -.6in \advance\evensidemargin by -.6in
\newtheorem*{cor}{Corollary}
\newtheorem*{lem}{Lemma}
\newtheorem*{lemma}{Lemma}
\newtheorem*{prop}{Proposition}
\newtheorem*{proposition}{Proposition}

\theoremstyle{definition}

\theoremstyle{definition}
\newtheorem*{thm}{Theorem}

\newtheorem*{rem}{Remark}

\newenvironment{pf}{\proof}{\endproof}
\newcounter{cnt}
\newenvironment{enumerit}{\begin{list}{{\hfill\rm(\roman{cnt})\hfill}}{%
\settowidth{\labelwidth}{{\rm(iv)}}\leftmargin=\labelwidth%
\advance\leftmargin by \labelsep\rightmargin=0pt\usecounter{cnt}}}{\end{list}} \makeatletter
\def\mydggeometry{\makeatletter\dg@YGRID=1\dg@XGRID=20\unitlength=0.003pt\makeatother}
\makeatother \theoremstyle{remark}

\numberwithin{equation}{section}

\let\bwdg\bigwedge
\def\bigwedge{{\textstyle\bwdg}}

\begin{document}

\newcommand{\thmref}[1]{Theorem~\ref{#1}}
\newcommand{\secref}[1]{Section~\ref{#1}}
\newcommand{\lemref}[1]{Lemma~\ref{#1}}
\newcommand{\propref}[1]{Proposition~\ref{#1}}
\newcommand{\corref}[1]{Corollary~\ref{#1}}
\newcommand{\remref}[1]{Remark~\ref{#1}}
\newcommand{\defref}[1]{Definition~\ref{#1}}
\newcommand{\er}[1]{(\ref{#1})}
\newcommand{\id}{\operatorname{id}}
\newcommand{\ord}{\operatorname{\emph{ord}}}
\newcommand{\sgn}{\operatorname{sgn}}
\newcommand{\wt}{\operatorname{wt}}
\newcommand{\tensor}{\otimes}
\newcommand{\from}{\leftarrow}
\newcommand{\nc}{\newcommand}
\newcommand{\rnc}{\renewcommand}
\newcommand{\dist}{\operatorname{dist}}

\nc{\cal}{\mathcal} \nc{\goth}{\mathfrak} \rnc{\bold}{\mathbf}
\renewcommand{\frak}{\mathfrak}
\newcommand{\supp}{\operatorname{supp}}
\newcommand{\Irr}{\operatorname{Irr}}
\renewcommand{\Bbb}{\mathbb}
\nc\bomega{{\mbox{\boldmath $\omega$}}} \nc\bpsi{{\mbox{\boldmath $\Psi$}}}
 \nc\balpha{{\mbox{\boldmath $\alpha$}}}
 \nc\bpi{{\mbox{\boldmath $\pi$}}}
\nc\bsigma{{\mbox{\boldmath $\sigma$}}} \nc\bcN{{\mbox{\boldmath $\cal{N}$}}} \nc\bcm{{\mbox{\boldmath $\cal{M}$}}} \nc\bLambda{{\mbox{\boldmath
$\Lambda$}}}

\newcommand{\loc}{\text{loc}}
\newcommand{\lie}[1]{\mathfrak{#1}}
\newcommand{\res}{\text{res}}

\newcommand{\tlie}[1]{\tilde{\mathfrak{#1}}}
\newcommand{\hlie}[1]{\hat{\mathfrak{#1}}}
\newcommand{\tscr}[1]{\tilde{\mathscr{#1}}}
\newcommand{\hscr}[1]{\hat{\mathscr{#1}}}
\newcommand{\hcal}[1]{\hat{\mathcal{#1}}}
\newcommand{\tcal}[1]{\tilde{\mathcal{#1}}}

\newcommand{\qmc}[1]{F_{#1}}
\newcommand{\qmcerr}[1]{E_{#1}}
\newcommand{\numm}[2]{\nu(#1,#2)}
\newcommand{\even}[1]{{#1}^+}
\newcommand{\odd}[1]{{#1}^-}
\newcommand{\poch}[3][q]{\left( #2 ; \, #1\right)_{#3}} 
\newcommand{\qbinom}[2]{{\genfrac[]{0pt}{}{#1}{#2}}_q}  
\newcommand{\complex}{\mathbb{C}}
\newcommand{\qmtil}[3]{\widetilde{A}_{#3}^{#1 \to #2}}
\newcommand{\qm}[3]{{A}_{#3}^{#1 \rightarrow #2}}
\newcommand{\floor}[1]{\left\lfloor #1 \right\rfloor}
\newcommand{\ceil}[1]{\lceil #1 \rceil}
\newcommand{\integers}{\mathbb{Z}}
\newcommand{\bset}[1][s]{\mathbb{B}^{(#1)}}
\newcommand{\gfunc}[1][s]{G^{(#1)}(x)}
\newcommand{\tpr}{t^\prime}
\newcommand{\tone}{t_1}
\newcommand{\tonepr}{\tone^\prime}
\newcommand{\beq}{\begin{equation}}
\newcommand{\eeq}{\end{equation}}
\newcommand{\vp}[1][p]{v^{(#1)}}
\newcommand{\pseries}[1][x]{\complex\!\left[ \left[#1\right]\right]}
\newcommand{\mcb}{\mathcal{B}}
\newcommand{\rbar}{\overline{r}}
\newcommand{\sdash}{{s^\prime}}
\newcommand{\zqi}[1][q,q^{-1}]{\integers\!\left[ #1 \right]}

\makeatletter
\def\section{\def\@secnumfont{\mdseries}\@startsection{section}{1}%
  \z@{.7\linespacing\@plus\linespacing}{.5\linespacing}%
  {\normalfont\scshape\centering}}
\def\subsection{\def\@secnumfont{\bfseries}\@startsection{subsection}{2}%
  {\parindent}{.5\linespacing\@plus.7\linespacing}{-.5em}%
  {\normalfont\bfseries}}
\makeatother
\def\subl#1{\subsection{}\label{#1}}
 \nc{\Hom}{\operatorname{Hom}}
  \nc{\mode}{\operatorname{mod}}
\nc{\End}{\operatorname{End}} \nc{\wh}[1]{\widehat{#1}} \nc{\Ext}{\operatorname{Ext}} \nc{\ch}{\operatorname{ch}} \nc{\ev}{\operatorname{ev}}
\nc{\Ob}{\operatorname{Ob}} \nc{\soc}{\operatorname{soc}} \nc{\rad}{\operatorname{rad}} \nc{\head}{\operatorname{head}}
\def\Im{\operatorname{Im}}
\def\gr{\operatorname{gr}}
\def\mult{\operatorname{mult}}
\def\Max{\operatorname{Max}}
\def\ann{\operatorname{Ann}}
\def\sym{\operatorname{sym}}
\def\Res{\operatorname{\br^\lambda_A}}
\def\und{\underline}
\def\Lietg{$A_k(\lie{g})(\bsigma,r)$}
\def\loc{\operatorname{loc}}

 \nc{\Cal}{\cal} \nc{\Xp}[1]{X^+(#1)} \nc{\Xm}[1]{X^-(#1)}
\nc{\on}{\operatorname} \nc{\Z}{{\bold Z}} \nc{\J}{{\cal J}} \nc{\C}{{\bold C}} \nc{\Q}{{\bold Q}}
\renewcommand{\P}{{\cal P}}
\nc{\N}{{\Bbb N}} \nc\boa{\bold a} \nc\bob{\bold b} \nc\boc{\bold c} \nc\bod{\bold d} \nc\boe{\bold e} \nc\bof{\bold f} \nc\bog{\bold g}
\nc\boh{\bold h} \nc\boi{\bold i} \nc\boj{\bold j} \nc\bok{\bold k} \nc\bol{\bold l} \nc\bom{\bold m} \nc\bon{\bold n} \nc\boo{\bold o}
\nc\bop{\bold p} \nc\boq{\bold q} \nc\bor{\bold r} \nc\bos{\bold s} \nc\boT{\bold t} \nc\boF{\bold F} \nc\bou{\bold u} \nc\bov{\bold v}
\nc\bow{\bold w} \nc\boz{\bold z} \nc\boy{\bold y} \nc\ba{\bold A} \nc\bb{\bold B} \nc\bc{\mathbb  C} \nc\bd{\bold D} \nc\be{\bold E} \nc\bg{\bold
G} \nc\bh{\bold H} \nc\bi{\bold I} \nc\bj{\bold J} \nc\bk{\bold K} \nc\bl{\bold L} \nc\bm{\bold M} \nc\bn{\mathbfb N} \nc\bo{\bold O} \nc\bp{\bold
P} \nc\bq{\bold Q} \nc\br{\bold R} \nc\bs{\bold S} \nc\bt{\bold T} \nc\bu{\bold U} \nc\bv{\bold V} \nc\bw{\bold W} \nc\bz{\mathbb  Z} \nc\bx{\bold
x} \nc\KR{\bold{KR}} \nc\rk{\bold{rk}} \nc\het{\text{ht }}

\nc\toa{\tilde a} \nc\tob{\tilde b} \nc\toc{\tilde c} \nc\tod{\tilde d} \nc\toe{\tilde e} \nc\tof{\tilde f} \nc\tog{\tilde g} \nc\toh{\tilde h}
\nc\toi{\tilde i} \nc\toj{\tilde j} \nc\tok{\tilde k} \nc\tol{\tilde l} \nc\tom{\tilde m} \nc\ton{\tilde n} \nc\too{\tilde o} \nc\toq{\tilde q}
\nc\tor{\tilde r} \nc\tos{\tilde s} \nc\toT{\tilde t} \nc\tou{\tilde u} \nc\tov{\tilde v} \nc\tow{\tilde w} \nc\toz{\tilde z}
\title[Demazure flags, Chebyshev polynomials,   Partial and  Mock Theta functions]{Demazure flags, Chebyshev polynomials, \\  Partial and  Mock Theta functions}
\author[Biswal, Chari, Schneider and Viswanath]{Rekha Biswal,  Vyjayanthi Chari, Lisa Schneider and Sankaran Viswanath}
\thanks{}

\address{The Institute of Mathematical Sciences, Chennai, India}
\email{rekha@imsc.res.in, svis@imsc.res.in}
\address{Department of Mathematics, University of California, Riverside, CA 92521, U.S.A.}
\email{vyjayanthi.chari@ucr.edu, lschn005@ucr.edu}

\thanks{V.C. was partially supported by  DMS-1303052. S.V acknowledges support from DAE under a XII
plan project.}

\maketitle
\begin{abstract} We  study the  level $m$--Demazure flag of a level $\ell$--Demazure module for $\frak{sl}_2[t]$.  We define the generating series $\qm{\ell}{m}{n}(x,q)$  which encodes the $q$--multiplicity of the level $m$ Demazure module of weight $n$.   We  establish two recursive formulae  for these functions.  We show that  the specialization to $q=1$     is a rational function involving the Chebyshev polynomials. We  give a   closed form for $\qm{\ell}{\ell+1}{n}(x,q)$ and prove that it is  given by a rational function.  In the case when  $m=\ell+1$ and $\ell=1,2$,  we  relate the generating series to partial theta series. We also study the specializations $\qm{1}{3}{n}(q^k,q)$ and relate them to the fifth order mock-theta functions of Ramanujan.
\end{abstract}

\section*{Introduction}

In this paper, we are interested in a  family of Demazure modules which occur in a highest weight integrable representation of the  affine  Lie algebra associated to $\lie{sl}_2$.  These Demazure modules are stable under the action of $\lie{sl}_2$;  in other words they are modules for  the current algebra $\lie{sl}_2[t]$ which is defined to be the Lie algebra  of polynomial  maps from $\bc$ to $\lie{sl}_2$. Alternatively, the current algbera is a maximal parabolic subalgebra of the affine Lie algebra.  The action of the element $d$ of the affine Lie algebra defines an integer grading on the current algebra and also a  compatible grading on the $\lie{sl}_2$-stable  Demazure modules.  In the rest of the paper, the term {\em Demazure module} will always mean a $\lie{sl}_2$-stable Demazure module.

The  Demazure modules are indexed by triples $(\ell, n,r)$ where  $n\in\bz_+$,  $r\in\bz$ and  $\ell \in \mathbb{N}$ and are denoted as $\tau_r^* D(\ell,n)$. The integer $\ell$ is called the level of the Demazure module and is given the action of the canonical central element of the affine algebra and $r\in\bz$ is minimal so that the corresponding graded component is non--zero.
A key result due to Naoi \cite{Naoi} states that if  $m \geq \ell
\ge 1$ then Demazure module $D(\ell,n)$ admits a filtration such that the successive quotients  are isomorphic to  level $m$ Demazure modules. In fact  Naoi  proves this result for an affine Lie algebra associated to a  simply--laced simple Lie algebra.  His proof is indirect using results of  \cite{Joseph} and\cite{Lusztig}.

A  direct and constructive proof of Naoi's result was obtained in \cite{CSSW} for $\lie{sl}_2$.  The methods of this paper also showed  the existence of a level $m$ Demazure flag in  a much wider class of modules for $\lie{sl}_2[t]$. As a result,
 explicit recurrence relations were given for the multiplicity of a  level $(\ell+1)$--Demazure module ocurring  in a   filtration of $\tau_r^*D(\ell, n)$.
A closed form solution of these recurrences  was however,  only obtained in some special cases: the numerical multiplicities (the $q=1$ case) were computed for $\ell = 2, m=3$, and the $q$-multiplicities for $\ell =1, m=2$.

 In this paper, we greatly extend the results of \cite{CSSW}.  We prove that  the generating function for the numerical multiplicity when $\ell=1$ is a  a rational function involving the Chebyshev polynomials.  A  level one Demazure module is isomorphic to a  {\em local Weyl module} \cite{CL} and hence our   result completely determines the numerical multiplicities  of a level $m$ flag of a local Weyl module for any given $m \geq 1$.
Our next main result concerns the $q$-multiplicities when $\ell =1$ and $m=3$. In this case, we first
show that the generating series can be written in terms of partial theta functions. Further, when
appropriately specialized, they reduce to expressions involving  the fifth order mock theta functions
$\phi_0, \phi_1, \psi_0, \psi_1$ of Ramanujan.
 The appearance of Ramanujan's mock theta functions in this set-up is quite unexpected and intriguing.
Certain {\em Hecke type double sums}, which are closely related objects, have previously appeared in Kac-Peterson's work \cite{KP} on characters of integrable representations of $\widehat{\lie{sl}_2}$. Further, mock theta functions (in the modern sense, following Zwegers \cite{Zwegers}) appear in Kac-Wakimoto's theory of affine superalgebras and their characters \cite{KW}.

We turn now to the  overall organization of  this paper.   We have arranged it so that the combinatorial results can be read essentially independently of the representation theory of $\lie{sl}_2[t]$.
In Section 1, we inroduce briefly the notion of a Demazure flag and define  the generating series $\qm{\ell}{m}{n}(x,q)$. We then  state the main combinatorial results of this paper.   In Section 2, we state the main representation theoretic results that are needed for the combinatorial study. The results of Section 2 can also be viewed as giving two  equivalent   definitions of $\qm{\ell}{m}{n}(x,q)$. It is far from obvious that these two definitions are equivalent and the proof of this, is representation theoretic and    can be found in Section 5 and Section 6.  In Section 3, we use the first definition of $\qm{\ell}{m}{n}(x, q)$  and study its specialization to $q=1$.   Section 4 uses the second definition to the  study the relationship with partial theta and mock theta functions.

{\em Acknowledgements. Lisa Schneider thanks Ole Warnaar for his very generous and invaluable help at an early stage of this work. Rekha Biswal thanks Travis Scrimshaw for his help in the writing of programs in SAGE used in the early stage of this work.}
\section{The main results}\label{mainresult} In this section, we give a concise description of the main results of this paper.  We keep the notation to a minimum and refer the reader to the later sections for precise definitions.
\subsection{}
Throughout this paper we denote by $\mathbb{C}$ the field of complex numbers and by $\mathbb{Z}$ (resp. $\mathbb{Z}_{+}$, $\mathbb{N}$) the subset of integers (resp. non-negative, positive integers). Given $n \in \bz_+$ and $m\in\bz$, set
\begin{gather*}\qbinom{n}{m}=\frac{(1-q^n)...(1-q^{n-m+1})}{(1-q)...(1-q^m)}, \ \ \  m>0 ,\\  \\ \qbinom{n}{0}=1,\ \  \qbinom{n}{m}= 0,\ \  m<0.\end{gather*}

\subsection{Demazure Flags and generating series}  Let $\lie{sl}_2[t]\cong\lie{sl}_2\otimes \bc[t]$ be the Lie  algebra of two by two matrices of trace zero with entries in  the algebra $\bc[t]$ of polynomials with complex coefficients  in an indeterminate $t$. The degree grading of $\bc[t]$ defines a natural grading on $\lie{sl}_2[t]$.   Let $D(\ell,s)$ be the $\lie{sl}_2[t]$--module generated by an element $v_s$ with  defining relations: \begin{gather}\label{locweyla}(x\otimes \bc[t])v_s=0,\ \ (h\otimes f) v_s= sf(0)v_s,\ \  \  (y\otimes 1)^{s+1}v_s=0,\\ \label{demrel2a}  (y\otimes t^{s_1+1}) v_s=0,\ \ \ (y\otimes t^{s_1})^{s_0+1} v_s=0,\ \ \ {\rm {if}}\ \  s_0<\ell.\end{gather}  
Here, $x,h,y$ is the standard basis of $\lie{sl}_2$ and $s_0\in\mathbb N$ and $s_1\in\bz$ with $s_1\ge -1$ and $s_0\le \ell$ are such that  $s=\ell s_1+s_0$. 
These modules are finite--dimensional and  $\ell$ is called the level of the Demazure module.  We refer the reader to Section \ref{proofofsecondrecursion} for the connection with the more traditional definition of the Demazure modules.

\medskip

 It was  observed in \cite{Naoi} that one could use the results of \cite{Joseph} and \cite{Lusztig} to show the following: for all integers $m\ge \ell>0$ and any non--negative integer $s$,   the module $D(\ell,s)$ admits a Demazure flag of level $m$,  i.e., there exists  a decreasing sequence of graded submodules of $D(\ell,s)$ such that the successive quotients of the flag are isomorphic to $\tau_p^*D(m, n)$ where $p\ge 0$, $0\le n\le  s$ and $s-n$ is even.  The number of times a particular level $m$--Demazure modul appears as a quotient in a level $m$--flag  is  independent of the choice of the  flag and we define a polynomial in an indeterminate $q$ by,
$$[D(\ell,s): D(m,n)]_q=\sum_{p\ge 0}[D(\ell,s):\tau_p^* D(m,n)]\,q^p,$$ where  $[D(\ell,s):\tau_p^* D(m,n)]$ is the multiplicity of $\tau_p^* D(m,n)$  in a level $m$--Demazure flag of $D(\ell,s)$. It  is known that $$[D(\ell,s): D(m,s)]_q=1,\qquad\ \ [ D(\ell, s): D(m,n)]_q=0\ \ \ s-n\notin2\bz_+.$$   Moreover,
for $m \ge \ell' \ge \ell$ we have
\begin{equation}\label{difflevel}
[D(\ell,s): D(m,n)]_q=\sum_{p\in\bz_{\ge 0}} [D(\ell,s): D(\ell', p)]_q\,[D(\ell',p): D(m,n)]_q.
\end{equation}

\medskip    Our primary goal in this paper is to understand both the polynomials $[D(\ell,s): D(m,n)]_q$ and   the associated generating series: given   $\ell, m \in \mathbb{N}$ with $m \geq \ell$, set
$$\qm{\ell}{m}{n}(x,q)=\sum_{k \geq 0}[D(\ell,n+2k):D(m,n)]_{q}\,x^k,\ \ \ n\ge 0. $$ It will be  convenient to set $\qm{1}{m}{-1}(x,1) = 1$.

\subsection{Numerical Multiplicity and Chebyshev Polynomials}\label{nm}  

Preliminary work using \cite{Sage} assisted in the formulation of the results in this section.
Our first result gives a recursive definition of $\qm{\ell}{m}{n}(x,q)$.
\begin{thm}\label{thm:num_mult_level_1m}
For  $n \geq -1$ and $m \geq 1$, the power series  $\qm{1}{m}{n}(x,1)$ satisfies the recurrence,
\begin{equation}\label{eq:lev1mrec}
\qm{1}{m}{n}(x,1)=
\begin{cases}
\qm{1}{m}{n+1}(x,1) - x\,\qm{1}{m}{n+2}(x,1) &\text{  if  } m \nmid n+2. \\
\\
\qm{1}{m}{n+1}(x,1) &\text{  if  } m \mid n+2.
\end{cases}
\end{equation}
\end{thm}
\medskip 
The proof of the theorem is in Section \ref{proofofnumerical}. We now discuss how to use the theorem to give
 a closed form for $ \qm{1}{m}{n}(x,1)$. We first  recall some relevant facts about Chebyshev polynomials. For $n \geq 0$, the Chebyshev polynomial $U_n(x)$ of the second kind, of degree $n$,  is  given by the recurrence relation:$$U_{n+1}(x) = 2xU_n(x) - U_{n-1}(x),\ \ U_0(x) = 1, \ \ U_1(x) = 2x.$$  It is known that the polynomials  $$P_n(x) = \sum_{k =0}^{\floor{\frac{n}{2}}} (-1)^k \,{n-k \choose k} \,x^k$$ satisfy
$$P_n(x^2) = x^n \,U_n(\,(2x)^{-1}) = \prod_{k=1}^n (1 - 2x\cos \frac{k \pi}{n+1}),$$ and also
\begin{equation}\label{eq:cheby_recc}
P_0 = P_1 = 1 \text{ and } P_{n+1}(x) = P_n(x) - xP_{n-1}(x) \text{ for } n \geq 1.
\end{equation} 
We now establish the following corollary of Theorem \ref{thm:num_mult_level_1m} which gives the closed form of $\qm{\ell}{m}{n}(x,1)$.

\begin{cor}\label{cor: num_mult_level_1m}
For   $n \in\bz_+$,  let $r, s$ be the unique non--negative integers such that $n = ms + r$ with  $0 \leq r < m$.  Then
$$ \qm{1}{m}{n}(x,1) = \frac{P_{m-r-1}(x)}{P_m(x)^{s+1}}.$$
\end{cor}
\begin{proof}
Set  $F_k = \qm{1}{m}{k}(x,1)$ for $k \ge -1$.  The corollary follows if we prove that for
all $k \geq 0$  and  $0\le p<m$, we have
$$\text{(a) } F_{mk+p}=P_{m-p-1}(x)\,F_{mk+m-1},\qquad \text{ (b) } F_{mk+m-1}=\frac{1}{P_m(x)^{k+1}}.$$
We first prove (a).
If $p = m-1$ this is immediate from the fact that $P_0(x)=1$, and if $p=m-2$ it follows from the second case in  \eqref{eq:lev1mrec}. Assume now that we have proved the equality for all $0 \leq p'<m$ with $p'>p$. To prove the equality for $p$ note that $m\nmid n+2$ and hence the first case of \eqref{eq:lev1mrec} applies. Together with the induction hypothesis and \eqref{eq:cheby_recc}, we get
$$F_{mk+p}= F_{mk+p+1}-x F_{mk+p+2} = (P_{m-p-2}(x)-x\,P_{m-p-3}(x))\,F_{mk+m-1}= P_{m-p-1}(x)\,F_{mk+m-1},$$
and the claim is established. To prove (b), observe that the first case of  \eqref{eq:lev1mrec} again, gives  $$F_{m(k-1)+m-1}= F_{mk}-xF_{mk+1}= (P_{m-1}(x)-x\,P_{m-2}(x))\,F_{mk+m-1}= P_m(x)\,F_{mk+m-1},\ \ k\ge 0.$$  Since $F_{-1}=1$ we get
 $P_m^{k+1}(x)\,F_{mk+m-1}=1$ and the proof of the corollary is complete.

\end{proof}
\medskip

More generally, in Section \ref{proofofnumerical} of this paper  we also study the series  $\qm{\ell}{m}{n}(x,1) $  and prove that they are rational functions in $x$ when $m = \ell+1$.

\subsection{ Fermionic Formulae } In certain special cases, it is possible to write down the polynomials $[D(\ell,s): D(m,n)]_q$ explicitly as  sums of products of $q$--binomials, i.e.,  by fermionic formulae.  If  $\ell=1$ and $m=2$, it was shown in \cite{CSSW}, that for all $k,n\in\bz_+$, we have \begin{equation}\label{1to2}\left[D(1, n+2k): D(2, n)\right]_q=  q^{k\lceil (n+2k)/2\rceil}{\qbinom{\lfloor(n+2k)/2\rfloor}{k}}. \end{equation}
\medskip

 In  Section \ref{qmultsec} of this paper  we shall prove that
\begin{prop}\label{2to3}
For  $r\in\{0,1,2\}$ and  $s\in\bz_+$ , set
$$\rbar =\begin{cases} 1\ \ \ r=1\\  0\ \ r=0,\,2 \end{cases}, \qquad \sdash =
\floor{\frac{s+1+\rbar}{2}}.$$
For all $p \in\bz_+$, we have
$$\left[ D(2, 3s+r+2p) : D(3, 3s+r)\right]_q = q^{\frac{1}{2}(p^2 +
  p(2s+r))} \sum_{\substack{j=0 \\ j \equiv p \!\!\!\!\!\pmod{2}}}^{\sdash}
q^{j(j-\rbar)/2} \qbinom{\frac{p-j}{2} + s}{s}
\,\qbinom{\sdash}{j}.$$
\end{prop}
Preliminary work using \cite{OEIS} assisted in the identification of the closed formulae in the proposition.
We now discuss several consequences of these formulae and we use freely the notation established so far.

\subsection{The functions $\qm{\ell}{\ell+1}{n}(x,q)$ for $\ell =1, 2$ and Partial Theta Functions}

Recall that  the partial
theta function and  the $q$-Pochammer symbol $\poch{a}{n}$ are given by,
$$\Theta(q,z) = \sum_{k= 0}^\infty q^{k^2}\,z^{k},\qquad  \poch{a}{n}= \prod_{i=1}^n (1-aq^{i-1}) ,\  n>0,\ \
\poch{a}{0}=1.$$
We refer the reader to \cite{AB} for more details regarding partial theta functions.
We now use the fermionic formulae to prove,
\begin{thm}Let $s\ge 0$.\begin{enumerit}
\item[(i)] For $r\in\{0,1\}$, we have\begin{equation}\label{eq: genf_12_final_form}
\qm{1}{2}{2s+r}(x,q)  
= \frac{1}{\poch{q}{s}} \,\sum_{i=0}^s (-1)^i q^{\frac{i(i+1)}{2}}
  \,\qbinom{s}{i} \, \Theta\left(q, \,xq^{i+s+r}\right).
\end{equation}
\item[(ii)]  For $r \in \{0, 1, 2\}$, we have
\begin{equation}
\qm{2}{3}{3s+r}(x,q) = \frac{1}{\poch{q}{s}} \sum_{i=0}^s \sum_{j=0}^\sdash
(-1)^i \,x^j \,q^{\beta(i,j)}
  \,\qbinom{s}{i} \, \qbinom{\, \sdash\,}{j} \Theta\left(q^2, \,x^2q^{\alpha(i,j)}\right),
\end{equation}
where \begin{gather*}\beta(i,j) = \frac{i(i+1)}{2} + j^2 + j\left(s + \frac{r -
  \rbar}{2}\right),\\ \alpha(i,j) = i +2j+2s + r.\end{gather*}
\end{enumerit}
\end{thm}
\begin{pf} Recall that for $n\ge 0$,   the {\em $q$-binomial theorem} states:
\beq\label{eq: qbthm_full}
\sum_{p=0}^n q^{\,p(p-1)/2}\,\qbinom{n}{p}\, x^p = \poch{-x}{n},
\eeq and hence we get
\begin{equation}\label{eq: qbt_step}
\qbinom{k + s}{s} = \frac{\poch{q^{k+1}}{s}}{\poch{q}{s}} =
\frac{1}{\poch{q}{s}}\sum_{i=0}^s (-q^k)^i \qbinom{s}{i}q^{i(i+1)/2}.
\end{equation} Equation \eqref{1to2} gives
\begin{equation}\label{eq: genf_12}
\qm{1}{2}{2s+r}(x,q) = \sum_{k=0}^\infty x^k q^{k(k+s+r)} \,
\qbinom{k+s}{s},
\end{equation}
for $s \geq 0$, $r \in \{0,1\}$ and using \eqref{eq: qbt_step} gives part (i).  The proof of (ii) is similar and we omit the  details.

\end{pf}
We remark here, that for  $\ell\ge 3$, the recursive formulae for $\qm{\ell}{\ell+1}{n}(x,q)$ are  very complicated and a solution seems difficult. However the preceding theorem does give some hints as to what form a solution might take.
\subsection{A closed form for $\qm{1}{3}{n}(x,q)$ and Mock theta functions} Using equation \eqref{difflevel} with $\ell=1$, $\ell'=2$ and $m=3$ and the formulae in \eqref{1to2} and Proposition \ref{2to3} we get:

\begin{equation}\label{eq: level13_explicit_form}
\qm{1}{3}{3s+r}(x,q) = \sum_{n=0}^\infty \,\sum_{p=0}^n
\!\!\!\sum_{\substack{j=0 \\ j \equiv p \\\!\!\!\!\!\pmod{2}}}^p x^n\,
q^{\frac{1}{2}\,\gamma(n,p,j)} \, \qbinom{n + \floor{\frac{3s+r}{2}}}{n-p}\, \qbinom{\frac{p-j}{2} + s}{s}
\,\qbinom{\sdash\,}{j}
\end{equation}
where $\gamma(n,p,j) = \left( n^2 + (n-p)^2 + j^2\right) +
n\,(2s+r)  + (n-p) \left(2\ceil{\frac{s-r}{2}} + r\right) +
j\,\left(-2\ceil{\frac{r}{2}} + r\right)$.
\medskip

We now discuss the relationship between certain specializations of  the  series $
\qm{1}{3}{n}(x,q)$ and the following  fifth order mock theta functions of Ramanujan \cite{Ramanujan,Watson}:
\begin{align}
\phi_0(q) &= \sum_{n=0}^\infty q^{n^2} \poch[q^2]{-q}{n}\label{eq: phi0},\\
\phi_1(q) &= \sum_{n=0}^\infty q^{(n+1)^2} \poch[q^2]{-q}{n}\label{eq: phi1},\\
\psi_0(q) &= \sum_{n=0}^\infty q^{\frac{(n+1)(n+2)}{2}} \poch{-q}{n}\label{eq: psi0},\\
\psi_1(q) &= \sum_{n=0}^\infty q^{\frac{n(n+1)}{2}} \poch{-q}{n}\label{eq: psi1}.
\end{align} Given any power series $f$ in the indeterminate $q$, we define
\beq \label{eq:pmdef}
\even{f}(q)= \sum_{n \geq 0} c_{2n} \,q^n = \frac{f(q^{\frac{1}{2}}) + f(-q^{\frac{1}{2}})}{2}, \;\;\;
  \odd{f}(q) = \sum_{n \geq 0} c_{2n+1} \,q^n = \frac{f(q^{\frac{1}{2}}) - f(-q^{\frac{1}{2}})}{2q^{\frac{1}{2}}},
\eeq
so that $ f(q) =  \even{f}(q^2) + q\,\odd{f}(q^2).$
We shall prove,
\begin{thm}\label{mocktheta} \begin{align*}
\qm{1}{3}{0}(1,q) &= \even{\phi}_0(q) &
\qm{1}{3}{0}(q,q) &= \odd{\phi}_1(q) \\
\qm{1}{3}{1}(1,q) &= \psi_1(q) &
\qm{1}{3}{1}(q,q) &= \psi_0(q)/q \\
\qm{1}{3}{2}(1,q) &= \odd{\phi}_0(q) &
\qm{1}{3}{2}(q,q) &= \even{\phi}_1(q) / q^2
\end{align*}
Moreover, for all $n \in \mathbb{Z}_+$ and $k \in \mathbb{Z}$, we have
$\poch{q}{\floor{\frac{n}{3}}} \,\qm{1}{3}{n}(q^k,q)$ is in the $\integers[q,q^{-1}]$-span of  $\{1, \phi^{\pm}_0, \phi^{\pm}_1, \psi_0, \psi_1\}$.
\end{thm}
\subsection{Some comments on the higher rank case} Assume that $\lie g$ is a simple Lie algebra of type $A$, $D$ or $E$ and let $\widehat{\lie g}$ be the associated affine Lie algebra.  In this case, the Demazure modules of interest  are indexed by triples  $(\ell,\lambda, r)$ where $\ell$ is a positive integer, $\lambda$ is a dominant integral weight for $\lie g$ and $r$ is an integer. The modules are denoted by $\tau_r^*D(\ell,\lambda)$.  It was shown in \cite{Naoi} that the modules $D(\ell,\lambda)$ admit a level $m$ Demazure flag if $m\ge \ell$ and so the polynomials $[D(\ell,\lambda): D(m,\mu)]_q$ are defined. As remarked earlier,  the  proof given in \cite{Naoi} does not lead to  recursive formulae. On the other hand, it is a non--trivial problem to generalize the methods of \cite{CSSW} to  the higher rank algebras:
see however \cite{Wand} for the level $1\to 2$ case for $\lie{sl}_{n+1}$.

\section{  Recursive formulae  for $[D(\ell,s): D(m,n)]_q$}\label{demflagrecursion}
In this section we    give two   recursive formulae for the polynomials $[D(\ell, s): D(m,n)]_q$,  both of which could be viewed as giving the definition of these polynomials.  It is far from obvious that these two definitions are equivalent. The proof of their equivalence is given in Sections \ref{proofofsecondrecursion} and Section \ref{proofofrecursion1} by showing that both recursions are satisfied by
 the multiplicities of the level $m$ Demazure flag in a level $\ell$ Demazure module.  The first recursive formula plays a critical role in studying $\qm{1}{\ell}{n}(x,1)$ while the second is essential in relating $\qm{\ell}{\ell+1}{n}(x,q)$  to the partial theta and mock theta functions.

\subsection{}  Given integers  $m \geq \ell>0$ and integers $s, n$, set
\begin{equation}\label{initial1}
[D(\ell,s): D(m,n)]_q=0,\  \ {\rm{if}} \ \ s<0\ \   {\rm{or}}\  n<0.
\end{equation}
We have
\begin{equation}\label{initial2}
[D(\ell,0): D(m,n)]_q=\delta_{n,0},\ \ n\in\bz_+,
\end{equation}
where $\delta_{j,k}$ is the Kronecker delta function.
More generally,
\begin{align}
[D(\ell,s): D(m,n)]_q &= 0,\  \ \text{ if } s-n\notin2\bz_+, \text{ and } \label{extra1}\\
[D(\ell,s): D(m,s)]_q &= 1,\ \ s\in\bz_+.\label{extra2}
\end{align}

\subsection{} Given a non--negative integer  $n$ and a positive integer $m$ let $0\le r(n,m)<m$ be the unique integer such that  $n=m\lfloor\frac{n}{m}\rfloor +r(n,m)$. The following result will be proved in Section \ref{proofofnumerical}.

\begin{thm}\label{recursion1} Let $\ell, m$ be positive integers with $m \geq \ell$.  For all $s,n \in \bz_+$, we have
\begin{gather*}
[D(\ell,s+1):D(m,n)]_q= [D(\ell,s):D(m,n-1)]_q+ (1-\delta_{r(n+1,m), 0})[D(\ell,s):D(m,n+1)]_q\\ - (1-\delta_{r(s,\ell),0})[D(\ell,s-1):D(m,n)]_q-q^{\lfloor\frac{s}{\ell}\rfloor \, r(s,\ell)}(1-q^{\lfloor\frac{s}{\ell}\rfloor})[D(\ell,s-2r(s,\ell)-1):D(m,n)]_q\\ + \ \ q^{(\lfloor\frac{n}{m}\rfloor+1)(m-r(n,m)-1)}(1-q^{\lfloor\frac{n}{m}\rfloor+1})[D(\ell,s):D(m,n+2m-2r(n,m)-1)]_q.
\end{gather*}

\end{thm}

\medskip
\begin{rem}   
The discussion so far  can be viewed as giving a recursive
definition of the polynomials  $[D(\ell,s):D(m,n)]_q$.  Thus,  \eqref{initial1} and \eqref{initial2} define
 $ [D(\ell,s):D(m,n)]_q$ for all $s \leq 0$ and $n \in \bz$.  For  $s \geq 0$, assume that we have defined $ [D(\ell, s'):D(m,n)]_q$ for all $s' \le s$ and all $n \in \bz$. The right hand side in Theorem \ref{recursion1} only involves $[D(\ell, s') : D(m, n')]$ with $s' \le s, n' \in \bz$ and hence  shows that  $[D(\ell,s+1):D(m,n)]_q$ is defined for all $n\in\bz_+$, and hence, by \eqref{initial1}, for all $n \in \bz$.
\end{rem}

\subsection{} In the case when $m=\ell+1$, we can prove a second recursion.
\begin{prop} \label{secondrecursion} Let $\ell$ be a positive integer.
\begin{enumerit}
\item[(i)] for  $0\le n, k \le \ell$, we have  $[D(\ell, k): D(\ell+1,n)]_q=\delta_{k,n}$ and  $$[D(\ell, 2\ell j\pm k) :  D(\ell+1, n)]_q=\delta_{k,n}\, q^{j(\ell j\,\pm\, n)},\ \ j\in\mathbb N.$$
\item[(ii)] if $n\ge \ell+1$ and  $s_0\in\mathbb N$  with $s_0\le \ell$ and $s_1\in\bz_+$, we have
\begin{eqnarray*}\label{eq: qmult_rec_cssw}
[D(\ell,\ell s_1+s_0):D(\ell+1,n)]_q &=&q^{(\ell s_1+s_0-n)/2}[D(\ell,\ell (s_1-1)+(s_0-1)):D(\ell+1,n-(\ell+1))]_q\\ &+& q^{s_0s_1}[D(\ell,\ell (s_1-1)+(\ell-s_0)):D(\ell+1,n)]_q.
\end{eqnarray*}
\medskip

\end{enumerit}

\end{prop}

Again, Equation \eqref{initial1} and Proposition \ref{secondrecursion} together give an inductive definition of $[D(\ell,s): D(\ell+1,n)]_q$. Part (i) of the proposition defines it for an $0\le n\le \ell$ once we note that any integer $s\ge 0$  is either of the form $2\ell j+k$ or $2\ell j -k$ for some $0\le k\le \ell$.  Part (ii) then defines it for $n\ge \ell+1$.   Together with the following assertion:
  for $m \ge \ell' \ge \ell$ we have
\begin{equation}\label{difflevel2}
[D(\ell,s): D(m,n)]_q=\sum_{p\in\bz_{\ge 0}} [D(\ell,s): D(\ell', p)]_q\, [D(\ell',p): D(m,n)]_q,
\end{equation} we  get an alternative definition of $ [D(\ell,s): D(m,n)]_q$. We emphasize that equation \eqref{difflevel2} is not obvious if we just use the definition of $[D(\ell,s): D(m,n)]_q$ from Theorem \ref{recursion1}, but it does become clear once we make the identification with multiplicities in a suitable  Demazure flag.

\section{   The  functions $\qm{\ell}{m}{n}(x,1)$}\label{proofofnumerical}
In this section  we  use Theorem \ref{recursion1} to analyze the functions $\qm{\ell}{m}{n}(x,1)$. Thus, we first  prove Theorem \ref{thm:num_mult_level_1m}. We then give closed formulae for these functions when $m=\ell+1$ in terms of certain initial conditions which are themselves given by recurrences. Finally, we discuss the general case of $\qm{\ell}{m}{n}(x,1)$.

\subsection{}\label{proof of thm1tom}  To prove  Theorem \ref{thm:num_mult_level_1m} we use Theorem \ref{recursion1} with $\ell=1$ and $q=1$. Since $r(p,1)=0$ for all $p\ge 0$, the recursion takes the following simpler form: for $n\ge -1$ and $k\ge 1$,

\begin{eqnarray*}
[D(1,n+1+2k):D(m,n+1)]_{q=1}&= &[D(1,n+2k):D(m,n)]_{q=1}\\ &+&(1-\delta_{r(n+1,m),m-1})[D(1,n+2k):D(m,n+2)]_{q=1}.\end{eqnarray*} Since $r(n+1,m)= m-1\iff m \mid n+2$,  we get
\begin{gather*}
[D(1,n+1+2k):D(m,n+1)]_{q=1}=\\ \begin{cases}[D(1,n+2k):D(m,n)]_{q=1}+[D(1,n+2k):D(m,n+2)]_{q=1}\ \ m\nmid n+2,\\ [D(1,n+2k):D(m,n)]_{q=1}\ \ m\mid n+2.\ \ \end{cases}\end{gather*}
Multiply both sides of the equation by $x^k$, sum  over $k\ge 1$ and  add one to both sides of the resulting equality of power series.
Recalling from \eqref{initial1} and \eqref{extra2} that  $[D(1, p): D(m,p)]_q=1$
and $[D(1,p): D(m,-1)]_q=0$ for all $p\ge 0$ now proves Theorem \ref{thm:num_mult_level_1m}.

\subsection{} We turn our attention to the study of $\qm{\ell}{\ell+1}{n}(x,1)$ for $\ell \ge 1$.
We prove,  \begin{thm}\label{thm:level_l_recc}
For $\ell \geq 1$ and  $n \geq 0$, write $n = (\ell+1)p_n- r_n$  where $p_n\in\bz_+$ and $0 \leq r_n \leq \ell$. Then,
\begin{equation}\label{eq:lev_l_recc}
\qm{\ell}{\ell+1}{n}(x,1)=
\begin{cases}
\qm{\ell}{\ell+1}{n+\ell}(x,1) - x^{r_n}\,\qm{\ell}{\ell+1}{n+2r_n}(x,1) &\text{  if  } \ell+1 \nmid n, \\
\\
\qm{\ell}{\ell+1}{n+\ell}(x,1) &\text{  if  } \ell+1 \mid n.
\end{cases}
\end{equation}
\end{thm}
\begin{rem}
Equation \eqref{eq:lev_l_recc} reduces to  \eqref{eq:lev1mrec} when $\ell=1, m=2, n \ge 0$. Thus, Theorem \ref{thm:level_l_recc} may be viewed as a generalization of this case of Theorem \ref{thm:num_mult_level_1m}.
\end{rem}

\subsection{} We shall use Theorem \ref{thm:level_l_recc} to establish the following result, which
in particular shows that the  functions $\qm{\ell}{\ell+1}{n}(x,1)$ are rational. For this we define polynomials $d_n$, $n\ge 0$  with non--negative integer coefficients as follows. Set
\medskip

$$ K_1 = \begin{bmatrix} 0 & 1 &&&& \\ &0&1&&& \\ &&0&1&& \\  &&&0&\ddots &\\ &&&& \ddots & 1\\ &&&&&0
   \end{bmatrix},  \;\;\;\;\;
K_2 = \begin{bmatrix}  &&&&& 0\\ &&&&\iddots &x^{\ell -1}\\ &&&\iddots & \iddots &\\
  &&0 & x^2 &&\\ & 0 & x &&&\\0 & 1 &&&& \end{bmatrix}
 \text{   and   } K = K_1 + K_2.$$

\medskip

 The polynomials $d_n$ are defined by requiring that the following equality hold for all $p\ge 0$: $$ \begin{bmatrix} d_{(\ell+1)p} & d_{(\ell+1)p+1} & \cdots &
  d_{(\ell+1)p+\ell  }\end{bmatrix}^{T}= K^{p+1} \begin{bmatrix} 1 & 1 & \cdots &
  1\end{bmatrix}^T.$$

\begin{prop}\label{closedform} Let $\ell \geq 1$. Then, for all $n \geq 0$, we have
  $$\qm{\ell}{\ell+1}{n}(x,1)=\frac{d_{n}}{(1-x^{\ell})^{\left\lfloor \frac{n}{\ell+1} \right\rfloor +1}}.$$

\end{prop}

\subsection{} {\em Proof of Theorem \ref{thm:level_l_recc}.}
To simplify notation, we fix $\ell \geq 1$, and for $s,n\in\bz_+$, set
$$ \numm{s}{n} := [D(\ell,s):D(\ell+1,n)]_{q=1}.$$
Recall that $\numm{s}{n} =0$ if $s < n$.
 The theorem follows if we prove that for all $s, n\ge 0$, we have
\beq\label{eq: nmultassert}
\numm{s}{n}= \begin{cases}
 \numm{s+\ell}{n+\ell}-\numm{s}{n+2r_n} &\text{if } (\ell+1) \nmid n.\\ \\
 \numm{s+\ell}{n+\ell} &\text{if } (\ell+1)  \mid  n.
  \end{cases}
\eeq
Notice that this equality holds whenever $s < n$, both sides being zero. Hence we have to prove it only in
the case when $s \ge n$.

Observe that   taking $q=1$ in Proposition \ref{secondrecursion}(i), gives
\beq\label{eq: facttwo}
s \ge 0, \; 0\le n\le \ell\implies
\numm{s}{n} = \begin{cases} 1 & \text{if } s+n \text{ or } s-n \text{ is a multiple of } 2\ell.\\
0 & \text{ otherwise}.
\end{cases}
\eeq
and taking $q=1$ in Proposition \ref{secondrecursion}(ii)   with $0<s_0\le \ell$, $s_1 >0$
and $s=\ell s_1+s_0$  gives
\beq\label{eq: factone} s\geq \ell+1, \; n\ge\ell+1\implies
\numm{s}{n} = \numm{s-\ell-1}{n-\ell-1} + \numm{s-2s_0}{n}.
\eeq
Observe in the last equation that $s-2s_0 = \ell s_1 - s_0 \geq 0$.

 We now proceed to prove \eqref{eq: nmultassert} by induction on $n$. To see that induction  begins, we  first prove that this assertion holds when $n=0$ for all $s\ge 0$. Using equation \eqref{eq: facttwo} it follows trivially  that $\numm{s}{0}=\numm{s+\ell}{\ell}$ as required.

Now let $n>0$. Assume that we have proved that $\numm{s}{n'}$ satisfies \eqref{eq: nmultassert} for all $0\le n'<n$ and for all $s\in\bz_+$.  We proceed by induction on $s$ to prove that   $\numm{s}{n}$  satisfies \eqref{eq: nmultassert}  for all $s\in\bz_+$. Notice that this induction begins at $s=0$ since both sides of  \eqref{eq: nmultassert} are then zero. Further, as remarked earlier, this equality holds for $s < n$; so we can further assume that $s \ge n$. Now, assume that we have proved the result for all $s'$ with $0\le s'<s$.  We have to consider two cases.

\noindent
{\em Case 1:} Suppose $0<n\le \ell$, and $s \geq n$. In this case we have $ n=\ell+1-r_n$ and we have  to prove that
$$\numm{s}{n}= \numm{s+\ell}{n+\ell}-\numm{s}{2\ell+2-n}.$$

{\em Case 1(a):} Suppose $s \geq \ell+1$. Then \eqref{eq: factone}  can be used for  both terms of the right hand side and we get
 \begin{align*}
\numm{s+\ell}{n+\ell}&=\numm{s-1}{n-1}+\numm{s+\ell-2s_0}{n+\ell},\\
\numm{s}{2\ell+2-n}&=\numm{s-\ell-1}{\ell+1-n}+\numm{s-2s_0}{2\ell+2-n}.
\end{align*}
Set
$$T_1 = \numm{s-1}{n-1} - \numm{s-\ell-1}{\ell+1-n}$$ and $$\qquad T_2 =  \numm{s+\ell-2s_0}{n+\ell} - \numm{s-2s_0}{2\ell+2-n}.$$

Equation   \eqref{eq: facttwo} applies to both the terms in $T_1$.  Now observing that:
\begin{align*}
(s-1) - (n-1) &= (s-\ell-1) + (\ell+1-n) \\
(s-1) + (n-1) &\equiv (s-\ell-1) - (\ell+1-n) \pmod{2\ell},
\end{align*}
we deduce that $T_1 =0$. Further, since $s-2s_0 < s$, the inductive hypothesis gives $T_2 = \numm{s-2s_0}{n}$.  We must thus prove that  $\numm{s}{n} = \numm{s-2s_0}{n}$. Since $s \equiv s_0 \pmod{\ell}$, we obtain $s - 2s_0 \equiv -s \pmod{2\ell}$; hence $s \pm n \equiv (s-2s_0) \mp n \pmod{2\ell}$; applying \eqref{eq: facttwo} completes the proof.

{\em Case 1(b):} Suppose $s \leq \ell$. Then since $2\ell + 2 -n > \ell$, we have
$\numm{s}{2\ell+2-n} =0$. We thus need to show that
$\numm{s}{n}= \numm{s+\ell}{n+\ell}$. Applying equation \eqref{eq: factone} again:
$$\numm{s+\ell}{n+\ell}= \numm{s-1}{n-1}+\numm{s+\ell-2s_0}{n+\ell}.$$
But since  $0 < s \leq \ell$, we have $s = s_0$, and hence $s+\ell-2s_0 < \ell < n+\ell$. Thus the second term vanishes. We need to now show that $\numm{s-1}{n-1} = \numm{s}{n}$. But from \eqref{eq: facttwo}, it is clear that for $1 \leq s,n \leq \ell$, $\numm{s-1}{n-1} = \numm{s}{n} = \delta_{s,n}$ .
This completes Case 1 of the inductive step.

\noindent
{\em Case 2:} Suppose $n \ge \ell+1$ and $s \ge n$.
Suppose first that $\ell + 1 \nmid n$. Consider
$$S = \numm{s+\ell}{n+\ell} - \numm{s}{n+2r_n} - \numm{s}{n}.$$ By applying \eqref{eq: factone} to each of these terms,
we have
\begin{align*}S &=\numm{s-1}{n-1}+\numm{s+\ell-2s_0}{n+\ell}-\numm{s-\ell-1}{n+2r_n-\ell-1}-\numm{s-2s_0}{n+2r_n}\\ &-\numm{s-\ell-1}{n-\ell-1}-\numm{s-2s_0}{n}.
\end{align*}
Since $n-\ell-1<n$ and $s-2s_0<s$, the inductive hypothesis gives
\begin{align*}
\numm{s-\ell-1}{n-\ell-1}&=\numm{s-1}{n-1}-\numm{s-\ell-1}{n+2r_n-\ell-1},\\
\numm{s-2s_0}{n}&=\numm{s+\ell-2s_0}{n+\ell}-\numm{s-2s_0}{n+2r_n}.
\end{align*}
Using these equations to replace $\numm{s-\ell-1}{n-\ell-1}$ and $\numm{s-2s_0}{n}$ in our equation for $S$, we obtain $S=0$ as required.

Now, suppose $\ell + 1 \mid n$. Consider $S'=\numm{s+\ell}{n+\ell}-\numm{s}{n}$. As in the case for $\ell + 1 \nmid n$, apply \eqref{eq: factone} to each term to get
$$S'=\numm{s-1}{n-1}+\numm{s+\ell-2s_0}{n+\ell}-\numm{s-\ell-1}{n-\ell-1}-\numm{s-2s_0}{n}.$$
Since $n-\ell-1<n$, $s-2s_0<s$ and $\ell + 1 \mid (n-\ell-1)$, the inductive hypothesis gives
$$
\numm{s-\ell-1}{n-\ell-1}=\numm{s-1}{n-1},\qquad \numm{s-2s_0}{n}=\numm{s+\ell-2s_0}{n+\ell}.
$$
This gives us $S'=0$ as required. \qed

\subsection{} {\em Proof of Proposition \ref{closedform}}
\begin{proof}
Let $n \geq 0$, with $n = (\ell+1)p_n - r_n$ and $0 \leq r_n \leq \ell$. We consider three cases in
equation \eqref{eq:lev_l_recc}: \\
(i) $1 \leq r_n\leq \ell-1$. In this case, define
$$n^\prime = n + 2r_n -\ell = (\ell+1)p_n - (\ell-r_n),$$
and consider $\qm{\ell}{\ell+1}{n}(x,1)$ and $\qm{\ell}{\ell+1}{n^\prime}(x,1)$.
Equation \eqref{eq:lev_l_recc} gives us the system of equations:
\begin{align*}
\qm{\ell}{\ell+1}{n}(x,1) &= \qm{\ell}{\ell+1}{n+\ell}(x,1) - x^{r_n} \,\qm{\ell}{\ell+1}{n^\prime + \ell}(x,1) \\
\qm{\ell}{\ell+1}{n^\prime}(x,1) &= \qm{\ell}{\ell+1}{n^\prime + \ell}(x,1) - x^{\ell-{r_n}} \,\qm{\ell}{\ell+1}{n + \ell}(x,1)
\end{align*}
(this becomes a single equation if $r_n = \ell-r_n$, i.e., if $n = n^\prime$).
Solving, we obtain:
\begin{equation} \label{eq: case1}
\qm{\ell}{\ell+1}{n + \ell}(x,1) = \frac{1}{1-x^\ell} \, \left( \qm{\ell}{\ell+1}{n}(x,1)  + x^{r_n}\,\qm{\ell}{\ell+1}{n^\prime}(x,1) \right)
\end{equation}

\medskip
\noindent
(ii) $r_n=0$, i.e., $n=(\ell+1)p_n$. Here we obtain $\qm{\ell}{\ell+1}{n + \ell}(x,1) = \qm{\ell}{\ell+1}{n}(x,1)$ \\
\medskip
(iii) $r_n = \ell$, i.e., $n=(\ell+1)p_n -\ell$. Then, $ \qm{\ell}{\ell+1}{n}(x,1) =  \qm{\ell}{\ell+1}{n + \ell}(x,1) - x^\ell \, \qm{\ell}{\ell+1}{(\ell+1)p_n+\ell}(x,1)$. Using case (ii) above, we obtain $\qm{\ell}{\ell+1}{n}(x,1) = (1-x^\ell) \qm{\ell}{\ell+1}{n+\ell}(x,1)$.

Now, for $n,p \ge 0$, define
$$d_n := \qm{\ell}{\ell+1}{n}(x,1) \cdot {(1-x^{\ell})^{\left\lfloor \frac{n}{\ell+1} \right\rfloor +1}},$$ and
$$ \zeta_p : = \begin{bmatrix} d_{(\ell+1)p} & d_{(\ell+1)p+1} & \cdots &
  d_{(\ell+1)p+\ell}  \end{bmatrix}^{T}.$$
We will prove by induction that $$\zeta_p=K^{p+1}\begin{bmatrix} 1 & 1 & \cdots &  1\end{bmatrix}^T$$ for $p\ge 0.$
When $p=0$, we use equation \eqref{eq: facttwo} to get
\[
d_n=(1-x^{\ell})\qm{\ell}{\ell+1}{n}(x,1) = \begin{cases} (1-x^{\ell})\sum_{k\ge 0}x^{\ell k},\ \ n=0, \ell\\ (1-x^{\ell})(\sum_{k\ge 0}x^{\ell k} + \sum_{k\ge 1}x^{\ell k -n}),\ \ 0<n<\ell.\end{cases}
\]
Thus, we have
\[
d_n=\begin{cases}1,\ \ n=0,\ell\\ 1+x^{\ell-n},\ \ 0<n<\ell.\end{cases}
\]
These polynomials are the entries in $\zeta_0$ and satisfy $\zeta_0=K\begin{bmatrix} 1 & 1 & \cdots &  1\end{bmatrix}^T.$
Now, let $p \geq 1$ and assume
\[
\zeta_{p-1}=K^{p}\begin{bmatrix} 1 & 1 & \cdots &  1\end{bmatrix}^T.
\]
We now have
$
K\zeta_{p-1}=K_1\zeta_{p-1}+K_2\zeta_{p-1},
$
where
\[
K_1\zeta_{p-1}=\begin{bmatrix} d_{(\ell+1)(p-1) + 1} & \cdots & d_{(\ell+1)(p-1) + \ell} & 0\end{bmatrix}^T,
\]
and
\[
K_2\,\zeta_{p-1}=\begin{bmatrix} 0& x^{\ell-1}d_{(\ell+1)(p-1) + \ell} & x^{\ell-2}d_{(\ell+1)(p-1) + \ell-1} & \cdots & d_{(\ell+1)(p-1) + 1}\end{bmatrix}^T.
\]
Dividing these vectors by $(1-x^{\ell})^{p+1}$, the equations \eqref{eq: case1} for $0<r<\ell$ and the cases for $r=0,\,\ell$ give us that $(K_1+K_2 )\zeta_{p-1}=\begin{bmatrix} d_{(\ell+1)p} & d_{(\ell+1)p+1} & \cdots & d_{(\ell+1)p+\ell}\end{bmatrix}=\zeta_p$. Then, by the inductive hypothesis, we have $$\zeta_p=K\zeta_{p-1}=KK^p\begin{bmatrix} 1 & 1 & \cdots &  1\end{bmatrix}^T$$ as desired.

\end{proof}

\subsection{}
Finally, we consider the general case, i.e., the multiplicities of
level $m$ Demazure modules in level $\ell$ Demazure modules for any $m
\geq \ell$. For $n \geq 0$, define
 $$ \qmtil{\ell}{m}{n}(x,q) = \sum_{s \geq 0} [D(\ell,s):D(m,n)]_{q}\,x^s.$$
Since the coefficient of $x^s$ is zero unless $s-n$ is a
non-negative even integer, we have $\qmtil{\ell}{m}{n}(x,q) = x^n\,\qm{\ell}{m}{n}(x^2,q)$ .
\begin{proposition}\label{prop: level_lm}
Let $1 \leq \ell \leq m$ and $n \geq 0$.
Let  $\beta_r(x) \in \pseries$, $0 \leq r < \ell$, be the unique power series such that
$$\qmtil{\ell}{m}{n}(x,1)=\sum_{r=0}^{\ell-1}\,
x^r\,\beta_{r}(x^{\ell}).$$
Then we have
$$\qmtil{1}{m}{n}(x,1)=\sum_{r=0}^{\ell-1} \,\qmtil{1}{\ell}{r}(x,1) \;\beta_{r}(y^{\ell}),$$
where $ y = {x}/{P_\ell(x^2)^\frac{1}{\ell}}$.
\end{proposition}
\begin{proof}
Let $\qmtil{\ell}{m}{n}(x,1)= \sum_{k=0}^\infty c_k x^k$.  For $k \geq 0$, letting
 $a(k), b(k)$ denote the unique integers such that $k = \ell a(k) + b(k)$ with $0 \leq b(k) < \ell$, we obtain
\beq \label{eq: anbn}
\beta_r(x) = \sum_{\{k: \,b(k) =r\}} c_k \, x^{a(k)}
\eeq
We now have
\begin{align}
\qmtil{1}{m}{n}(x,1) &= \sum_{s \geq 0}\left[D(1,s):D(m,n)\right]_{q=1}\,x^s = \sum_{s \geq 0}\sum_{u \geq 0}\left[D(1,s):D(\ell,u)\right]_{q=1} \, \left[D(\ell,u):D(m,n)\right]_{q=1} \,x^s
\notag\\
&= \sum_{u \geq 0} c_u \,\qmtil{1}{\ell}{u}(x,1) \label{eq: atcalc}
\end{align}
Corollary \ref{cor: num_mult_level_1m} implies that $\qmtil{1}{\ell}{u}(x,1) = \qmtil{1}{\ell}{b(u)}(x,1) \left[ \frac{x^\ell}{P_\ell(x^2)} \right]^{a(u)}$. Substituting this into equation \eqref{eq: atcalc}:
$$ \qmtil{1}{m}{n}(x,1) = \sum_{r=0}^{\ell-1} \, \qmtil{1}{\ell}{r}(x,1) \left( \sum_{\substack{u \geq 0 \\ b(u) = r}} c_u \left[ \frac{x^\ell}{P_\ell(x^2)} \right]^{a(u)} \right).$$
From equation \eqref{eq: anbn}, the inner sum is just $\beta_r(y^\ell)$ with $y = {x}/{P_\ell(x^2)^\frac{1}{\ell}}$, and the proof is complete.
\end{proof}

\medskip
\begin{cor}
Let $m \geq 2, n \geq 0$. Then $$\qm{2}{m}{n}(x,1)=\left(\frac{1}{1+x}\right)^{\left\lfloor \frac{n}{2} \right\rfloor +1}\,\qm{1}{m}{n}\left(\frac{x}{1+x},1\right).$$	
\end{cor}

\begin{proof}
This follows by taking $\ell = 2$ in Proposition \ref{prop: level_lm},
and rewriting everything in terms of the $A_n$.

\end{proof}

\begin{rem}
Fix $\ell \geq 1$.
Let $R$ denote the $\complex$-algebra $\pseries$, and $S$ be the
subalgebra $\pseries[x^\ell]$. Then, $R$ is a free $S$-module of rank
$\ell$. Further, for any units $u_0, u_1, \cdots, u_{\ell -1}$ in $R$,
the set $\{u_r \,x^r: 0 \leq r < \ell \}$ is an
$S$-basis of $R$. Consider the following two choices of basis:
$$\mcb_1= \{x^r: 0 \leq r <\ell\}; \;\;\; \mcb_2=\{\qmtil{1}{\ell}{r}(x,1): 0 \leq r < \ell\}.$$
The latter forms a basis since $\qmtil{1}{\ell}{r}(x,1) =
x^r\,\qm{1}{\ell}{r}(x^2,1)$ and $\qm{1}{\ell}{r}(x^2,1)$ is a unit in
$R$ since its constant term is 1. Now, the map
$$\phi: \pseries[y] \to \pseries \text{ defined by } y \mapsto
\frac{x}{P_\ell(x^2)^\frac{1}{\ell}}$$
is an isomorphism of algebras. Since $\phi^{-1}(x) = uy$ for some unit
$u \in \pseries[y]$, it clear that the pull-back $\mcb^\prime_2 = \{
\phi^{-1}(b): b \in \mcb_2\}$ of $\mcb_2$ is of the form $\{u_r \,
y^r: 0 \leq r < \ell\}$ for some units $u_r$ in $\pseries[y]$. Hence
$\mcb^\prime_2$ is a basis of $R^\prime = \pseries[y]$ over $S^\prime
= \pseries[y^\ell]$.

Now, suppose we are given $m \geq \ell$ and $n \geq 0$. To obtain the
generating series  $\qmtil{\ell}{m}{n}(x,1) \in R$, it is enough to
obtain its coordinates $\beta_r(x^\ell) \in S$ with respect to the
basis $\mcb_1$. Proposition \ref{prop: level_lm} gives us a way of
determining the $\beta_r$ (in principle). Consider $F = \qmtil{1}{m}{n}(x,1) \in R$; this
is known in closed form by Theorem \ref{thm:num_mult_level_1m}. The
coordinates of $F^\prime = \phi^{-1}(F) \in R^\prime$ with respect to
the basis $\mcb^\prime_2$ are precisely the $\beta_r(y^\ell)$.
\end{rem}

\section{The functions $\qm{\ell}{\ell+1}{n}(x,q)$ when $\ell=1,2$ and mock theta functions}\label{qmultsec}
In this section, we prove Proposition \ref{2to3} and Theorem \ref{mocktheta}.

\subsection{} We first use Proposition \ref{secondrecursion} to give closed formulae for $\qm{2}{3}{n}(x,q)$. In terms of generating series, Proposition \ref{secondrecursion}(i) gives, \begin{equation}\label{2to3initial}
\qm{2}{3}{0}(x,q) = \sum_{s\ge0} q^{2s^2}x^{2s},\ \ \ \
\qm{2}{3}{1}(x,q) =\sum_{s\ge 0}q^{\frac{s(s+1)}{2}}x^{s}, \ \ \ \
\qm{2}{3}{2}(x,q) = \sum_{s\ge 0}q^{2s(s+1)}x^{2s}.
\end{equation}
 and Proposition \ref{secondrecursion}(ii) gives for $k \geq 3$,
\begin{equation}\label{eq: brecc}
\qm{2}{3}{k-3}(xq,q)  = \begin{cases}
\qm{2}{3}{k}(x,q) - xq^{\frac{k+1}{2}}\,\qm{2}{3}{k}(xq, q) & \text{ if $k$ is
  odd.}\\
\\
\qm{2}{3}{k}(x,q) - x^2q^{k+2}\,\qm{2}{3}{k}(xq^2, q) &\text{   if $k$ is even.}
\end{cases}
\end{equation}
 We have the following result which solves this recurrence explicitly.
\begin{proposition} \label{prop: bnew_exp}
Let $r \in \{0, 1, 2\}$ and $s \geq 0$, and set $$\rbar =\begin{cases} 1,\ \ \ r=1,\\  0,\ \ r=0,\,2.\end{cases}$$
Then, we have
\begin{equation}\label{eq: bnew_exp}
\qm{2}{3}{3s + r}(x,q) = \sum_{p=0}^\infty x^p \,q^{\frac{1}{2}(p^2 +
  p(2s+r))} \sum_{\substack{j=0 \\ j \equiv p \!\!\!\!\!\pmod{2}}}^p
q^{j(j-\rbar)/2} \qbinom{\frac{p-j}{2} + s}{s}
\,\qbinom{\floor{\frac{s+1+\rbar}{2}}}{j}.
\end{equation}
\end{proposition}
\begin{proof}
We check the initial conditions first. Let $s=0$; in this case, the inner sum in equation \eqref{eq: bnew_exp} equals 1 if either (i) $\rbar =1$, or (ii) $\rbar=0$ and $p \equiv 0 \pmod{2}$, and is zero otherwise. From this, it is clear that \eqref{eq: bnew_exp} reduces to equations (\ref{2to3initial}) when $s=0$.

Next, for $s \geq 1$, we verify that the recurrence relation \eqref{eq: brecc} holds.
We set $\alpha(s) = \floor{\frac{s + \rbar +1}{2}}$ and $\beta(j,p,s) = \frac{p-j}{2} + s$.

First, suppose $k = 3s+r$ is even. Then, $s + \rbar$ is even, and the recurrence in
Equation \eqref{eq: brecc} is equivalent to the statement that the following sum vanishes for all $p \geq 0$:
\beq \label{eq: evenrecstat}
\sum_{\substack{j=0 \\ j \equiv p \!\!\!\!\!\pmod{2}}}^p
\!\!\!\! q^{j(j-\rbar)/2} \qbinom{\alpha(s)}{j} \left( \qbinom{\beta(j,p,s)}{s} - \qbinom{\beta(j,p,s-1)}{s-1} - q^s\,\qbinom{\beta(j,p-1,s)}{s}\right).
\eeq
Notice that $\beta(j,p,s-1)=\beta(j,p-1,s)=\beta(j,p,s)-1$.
But, from the $q$-binomial identity
\beq\label{eq: qbid}
\qbinom{a}{b} = \qbinom{a-1}{b}+q^{a-b}\, \qbinom{a-1}{b-1}
\eeq
with $a = \beta(j,p,s), \,b=\beta(j,p,s)-s$, we see that each summand in \eqref{eq: evenrecstat} is in fact zero. This proves the recurrence relation for $k$ even.

Next, let $k = 3s + r$ be odd. In this case, the recurrence relation of Equation \eqref{eq: brecc} is equivalent to the statement that the following sum vanishes for all $p\geq 0$:
\beq\label{eq: 23exp_oddsimp}
\!\!\!\! \sum_{\substack{j=0 \\ j \equiv p \!\!\!\!\!\!\pmod{2}}}^p \!\!\!\!\!\!q^{j(j-\rbar)/2} \left( \qbinom{\beta(j,p,s)}{s}\qbinom{\alpha(s)}{j}-q^{\alpha(s)-j}\qbinom{\beta(j,p,s)}{s}\qbinom{\alpha(s)}{j-1} -\qbinom{\beta(j,p,s-1)}{s-1}\qbinom{\alpha(s-1)}{j}\right).
\eeq

Notice that $\beta(j,p,s-1)=\beta(j,p,s)-1$ and $\alpha(s-1)=\alpha(s)-1$ since $s+\rbar$ is odd.
Using the identity \eqref{eq: qbid} twice in succession, first with $(a,b)=(\alpha(s), j)$ and then with $(a,b) = (\alpha(s), j-1)$, we obtain:
\beq \label{eq: intermediate_1}
\qbinom{\alpha(s)}{j} = \qbinom{\alpha(s)-1}{j} + q^{\alpha(s)-j}\,\qbinom{\alpha(s)}{j-1} - q^{2(\alpha(s)-j)+1}\, \qbinom{\alpha(s)-1}{j-2}.
\eeq
Similarly, choosing $a = \beta(j,p,s)$ and $b = \beta(j,p,s)-s$ in \eqref{eq: qbid} gives:
\beq\label{eq: intermediate_2}
\qbinom{\beta(j,p,s)-1}{s-1} = \qbinom{\beta(j,p,s)}{s} -q^s\, \qbinom{\beta(j,p,s)-1}{s}.
\eeq
Substituting Equations \eqref{eq: intermediate_1}, \eqref{eq: intermediate_2} into the first and third terms of \eqref{eq: 23exp_oddsimp} respectively, and simplifying, the expression in \eqref{eq: 23exp_oddsimp} becomes
$$
\sum_{\substack{j=0 \\ j \equiv p \!\!\!\!\!\pmod{2}}}^p \!\!\! q^{j(j-\rbar)/2} \,
\qbinom{\beta(j,p,s)-1}{s} \,\qbinom{\alpha(s)-1}{j} -
\sum_{\substack{j=0 \\ j \equiv p \!\!\!\!\!\pmod{2}}}^p \!\!\! q^{(j-2)(j-\rbar-2)/2} \,
\qbinom{\beta(j,p,s)}{s} \,\qbinom{\alpha(s)-1}{j-2}.
$$
Re-indexing the second sum with $j^\prime = j-2$ proves this expression is zero. This completes the proof.
\end{proof}
\subsection{} We are now able to deduce Proposition \ref{2to3}. We define $s'=\lfloor\frac{s+1+\rbar}{2}\rfloor$. We first note that for $j> \min(p,s')$, $$\qbinom{\frac{p-j}{2}+s}{s}\qbinom{s'}{j}=0,$$ and thus we can take the inner summation in \eqref{eq: bnew_exp} from $j=0$ to $j=s'$ with $j\equiv p\!\!\pmod {2}$. Extracting the coefficient of $x^p$ in equation (\ref{eq: bnew_exp}), we obtain the explicit polynomial for $[D(2,3s+r+2p):D(3,3s+r)]_q$ in Proposition \ref{2to3}.

\subsection{} Now, we are able to deduce equation (\ref{eq: level13_explicit_form}).
For $n\ge 0$, we have the coefficient of $x^n$ in $\qm{1}{3}{3s+r}(x,q)$ is $[D(1,3s+r+2n):D(3,3s+r)]_q$. Using equation (\ref{difflevel}), we have
\begin{gather*}
[D(1,3s+r+2n):D(3,3s+r)]_q\\ =\sum_{p=0}^n[D(1,3s+r+2n):D(2,3s+r+2p)]_q \, [D(1,3s+r+2p):D(3,3s+r)]_q.
\end{gather*}
We use equation (\ref{1to2}) and Proposition \ref{2to3} to obtain the explicit form of the coefficient of $x^n$ as stated in \eqref{eq: level13_explicit_form}.

\subsection{}We isolate the formulae for $\qm{1}{3}{3s+r}(x,q)$ for $3s+r=0,1,2$. In equation \eqref{eq: level13_explicit_form}, when $3s+r=0$, we have $s=r=\rbar=\sdash=0$ and hence $\qbinom{\sdash}{j}=0$ unless $j=0$. We also have $\gamma (n,p,0)=n^2+(n-p)^2$. Reindexing by $p\mapsto n-p$, we have
\begin{equation}\label{1to3initial0}\qm{1}{3}{0}(x,q) = \sum_{n=0}^\infty x^n \, q^{n^2/2}
\!\!\!\sum_{\substack{p=0 \\ p \equiv n \\\!\!\!\!\!\pmod{2}}}^n
q^{\,p^2/2} \, \qbinom{n}{p}.\end{equation}
For $3s+r=2$, we have $s=\rbar=\sdash=0$, $r=2$ and $\gamma (n,p,0)=n^2+(n-p)^2+2n-p$. Using the same reasoning in the previous case, we obtain
\begin{equation}\label{1to3initial2}
\qm{1}{3}{2}(x,q) = \sum_{n=0}^\infty x^n \, q^{\frac{n^2 + 2n}{2}}
\!\!\!\sum_{\substack{p=0 \\ p \equiv n \\\!\!\!\!\!\pmod{2}}}^n
q^{\frac{p^2}{2}} \, \qbinom{n+1}{p} = (xq^{\frac{1}{2}})^{-1}\sum_{n=1}^\infty x^{n} \, q^{\frac{n^2}{2}}
\!\!\!\sum_{\substack{p=0 \\ p \not\equiv n \\\!\!\!\!\!\pmod{2}}}^n
q^{\frac{p^2}{2}} \, \qbinom{n}{p}.
\end{equation}
Lastly, for $3s+r=1$, we have $s=0$ and $r=\rbar=\sdash=1$ thus $\qbinom{\sdash}{j}=0$ unless $j=0,1$. Since $\gamma(n,p,0)=n^2+(n-p)^2+2n-p=\gamma(n,p,1)$, we have
\begin{equation}\label{1to3initial1}
\qm{1}{3}{1}(x,q) = \sum_{n=0}^\infty x^n\,  q^{\frac{n(n+1)}{2}} \,\sum_{p=0}^n
q^{\frac{p(p+1)}{2}}\qbinom{\,n}{\,p}.
\end{equation}

\subsection{} For the rest of the section we shall be interested in the  specializations  $\qm{1}{3}{n}(q^k,q)$ for $k \in \mathbb{Z}$, $n\in\bz_+$. For this, it is convenient to define \begin{align}
\Phi(x,q) &= \sum_{n=0}^\infty x^n \,q^{n^2} \poch[q^2]{-q}{n}\label{eq: Phi_def}.\\
\Psi(x,q) &= \sum_{n=0}^\infty x^n\,  q^{\frac{n(n+1)}{2}} \poch{-q}{n}\label{eq: Psi_def}.
\end{align}
The following Lemma will be useful. \begin{lemma} \label{prop: phipsirec}
\begin{align}
\Psi(x,q) &=
xq^2 \,\Psi(xq^2, q) + xq \,\Psi(xq, q) + 1. \label{eq: psirec}\\
\Phi(x,q^\frac{1}{2}) &=
xq \,\Phi(xq^2, q^\frac{1}{2}) + xq^\frac{1}{2} \,\Phi(xq, q^\frac{1}{2}) + 1. \label{eq: phirec}
\end{align}
\end{lemma}
\begin{proof}
We will only prove \eqref{eq: psirec}, since \eqref{eq: phirec} is similar. From \eqref{eq: Psi_def}, it follows that the right hand side of \eqref{eq: psirec} is the following sum:
$$ 1 + \sum_{n=0}^\infty x^{n+1}\,\poch{-q}{n} \, q^{n(n+1)/2} \left( q^{2n+2} + q^{n+1}\right).$$
Reindexing this sum with $n^\prime = n+1$, it is clear that it equals $\Psi(x,q)$.
\end{proof}

\subsection{} We now prove,
\begin{prop} \beq\label{eq: level13_hwt0}
\qm{1}{3}{0}(x,q) = \frac{1}{2}\left( \Phi(x,q^{\frac{1}{2}}) + \Phi(x,-q^{\frac{1}{2}}) \right)
\eeq
\beq\label{eq: level13_hwt2}
\qm{1}{3}{2}(x,q) = \frac{1}{2xq^{\frac{1}{2}}}\left(
  \Phi(x,q^{\frac{1}{2}}) - \Phi(x,-q^{\frac{1}{2}}) \right)
\eeq
\beq\label{eq: level13_hwt1}
\qm{1}{3}{1}(x,q)
 = \Psi(x,q)
\eeq
\end{prop}
\begin{proof}
For $\qm{1}{3}{0}(x,q)$, we first use the $q$-binomial theorem (Equation \eqref{eq: qbthm_full}) to obtain
\begin{equation}\label{eq: qbteq}
\sum_{p=0}^n q^{\,p^2/2} \, \qbinom{n}{p}\, z^{n-p} =
z^n\,\poch{-z^{-1}q^{\frac{1}{2}}}{n}.
\end{equation}
We obtain a second equation by replacing $z$ by $-z$ in \eqref{eq: qbteq}. Adding these two equations together and setting $z=1$, we have
$$\sum_{\substack{p=0 \\ p \equiv n \\\!\!\!\!\!\pmod{2}}}^n q^{\,p^2/2} \, \qbinom{n}{p} =
\frac{1}{2}\poch{-q^{\frac{1}{2}}}{n} + \frac{(-1)^n}{2}\,\poch{q^{\frac{1}{2}}}{n}.   $$
Replacing this summation in (\ref{1to3initial0}), we obtain \eqref{eq: level13_hwt0}.

The proof for $\qm{1}{3}{2}(x,q)$ is similar. For $\qm{1}{3}{1}(x,q)$, apply the $q$-binomial theorem to
the inner sum in \eqref{1to3initial1} to obtain \eqref{eq: level13_hwt1}.
\end{proof}

\subsection{} We are now able to make the connection with mock theta functions and prove the first assertions of Theorem \ref{mocktheta}.
\begin{cor}\label{prop: x1q_spec} For the specializations $x=1$ and $x=q$, we have
\begin{align*}
\qm{1}{3}{0}(1,q) &= \even{\phi}_0(q) &
\qm{1}{3}{0}(q,q) &= \odd{\phi}_1(q) \\
\qm{1}{3}{1}(1,q) &= \psi_1(q) &
\qm{1}{3}{1}(q,q) &= \psi_0(q)/q \\
\qm{1}{3}{2}(1,q) &= \odd{\phi}_0(q) &
\qm{1}{3}{2}(q,q) &= \even{\phi}_1(q) / q^2
\end{align*}
\end{cor}
\begin{proof}
We note from equations \eqref{eq: Phi_def}, \eqref{eq: Psi_def},\eqref{eq: phi0}-\eqref{eq: psi1} that trivial calculations give
\begin{equation}\label{seriesequalities}
 \Psi(1,q) = \psi_1(q),\qquad \Psi(q,q) = \psi_0(q)/q,\qquad
\Phi(1,q) = \phi_0(q), \qquad \Phi(q^2,q) = \phi_1(q)/q.
\end{equation}
Since $\qm{1}{3}{1}(x,q)=\Psi(x,q)$ from \eqref{eq: level13_hwt1}, we easily obtain the equalities
$$
\qm{1}{3}{1}(1,q) = \psi_1(q) \qquad \text{ and }\qquad \qm{1}{3}{1}(q,q) = \psi_0(q)/q.
$$
Now, consider \eqref{eq: level13_hwt0} with the equation for $\Phi(1,q^{1/2})$ from \eqref{seriesequalities} above to obtain
$$
\qm{1}{3}{0}(1,q)=\frac{1}{2}(\phi_0(q^{1/2})+\phi_0(-q^{1/2})).
$$
Thus by \eqref{eq:pmdef}, we obtain $$\qm{1}{3}{0}(1,q)=\even{\phi}_0(q).$$
Similar calculations give $$\qm{1}{3}{0}(q,q)=\odd{\phi}_1(q).$$
For the last two equalities, we use Equations \eqref{eq: level13_hwt2} and \eqref{seriesequalities} and proceed as above.
\end{proof}

\subsection{}
We now consider the specializations $\qm{1}{3}{n}(q^k, q)$ for arbitrary $k \in \integers$ and
$0 \leq n \leq 2$. We show that these are in fact linear combinations of the mock theta functions with coefficients in $\zqi$. More precisely, we have

\begin{thm}\label{thm: xqarbit}
Let $k \in \integers$. Then:
\begin{enumerate}
\item $$\qm{1}{3}{1}(q^k, q) = a_{k,0}(q) \, \psi_0(q) + a_{k,1}(q) \, \psi_1(q) + b_k(q),$$
for some $a_{k,0}, a_{k,1}, b_k \in \zqi$.

\medskip
\item  $$\qm{1}{3}{0}(q^k, q) = c_{k,0}(q) \, \phi^{\pm}_0(q) + c_{k,1}(q) \, \phi^{\pm}_1(q) + d_k(q),$$
for some $c_{k,0}, c_{k,1}, d_k \in \zqi$. The choice of signs ($\pm$) on the right hand side is made
as follows: both signs are ($+$) if $k$ is even, and both are ($-$) if $k$ is odd.

\medskip
\item  $$\qm{1}{3}{2}(q^k, q) = e_{k,0}(q) \, \phi^{\pm}_0(q) + e_{k,1}(q) \, \phi^{\pm}_1(q) + f_k(q),$$
for some $e_{k,0}, e_{k,1}, f_k \in \zqi$. The choice of signs ($\pm$) on the right hand side is now opposite to that above, with both signs ($-$) if $k$ is even, and ($+$) if $k$ is odd.
\end{enumerate}
\end{thm}
\begin{proof}
All three assertions hold for $k=0,1$ by Proposition \ref{prop: x1q_spec}.
We first prove (1). Let $k \in \integers$;
equations \eqref{eq: level13_hwt1} and \eqref{eq: psirec} imply:
\beq\label{eq: a13_qpow_rec}
1 -\qm{1}{3}{1}(q^k,q) + q^{k+1} \,\qm{1}{3}{1}(q^{k+1}, q) +
q^{k+2} \,\qm{1}{3}{1}(q^{k+2}, q)   =0.
\eeq

\smallskip
Consider $\qm{1}{3}{1}(q^j, q)$ for $j \in \{k, k+1, k+2\}$; equation \eqref{eq: a13_qpow_rec} shows that if the assertion of the theorem holds for any two of these values of $j$, then it also holds for the third. Since, as observed earlier, the assertion is true for $k=0,1$, it holds for all $k \in \integers$ by induction.

To prove (2) and (3), we observe that equations  \eqref{eq: phirec}, \eqref{eq: level13_hwt0} and
  \eqref{eq: level13_hwt2}  imply:
\begin{align}
\qm{1}{3}{0}(x,q) &= xq\,\qm{1}{3}{0}(xq^2,q) +  x^2q^2\,\qm{1}{3}{2}(xq,q) +1. \\
\qm{1}{3}{2}(x,q) &= xq^3\,\qm{1}{3}{2}(xq^2,q) +  \qm{1}{3}{0}(xq,q).
\end{align}
The proof now follows by setting $x=q^k$, and arguing by induction as in (1).
\end{proof}

\subsection{}
Finally, we turn to $\qm{1}{3}{n}(x, q)$ for arbitrary $n \ge 0$.
Let us define
$$ \qmc{n}(x, q) = \qm{1}{3}{n}(x,q)\,\prod_{i =1}^{\floor{\frac{n}{3}}} (1-q^i),$$
with $\qmc{-1}(x,q)=0$. Let $\integers((q))$ denote the ring of Laurent series with integer coefficients.
We then have the following:
\begin{proposition} Let $R \subset \integers((q))$ denote the $\zqi$-span of $\{1, \phi^{\pm}_0, \phi^{\pm}_1, \psi_0, \psi_1\}$. Let $n \geq 0, \,k \in \integers$. Then $\qmc{n}(q^k, q) \in R$.
\end{proposition}
\begin{proof}
It is easy to check that the recursion for graded multiplicities obtained in Theorem \ref{recursion1} translates into the following relation for the generating series, valid for all $p \geq 1, r \in \{0,1,2\}$:

\begin{equation*}
q^{pr}x^{r+1}\,\qmc{3p+r}(x,q) = (1+x)\qmc{3p-r-1}(x,q)-\qmc{3p-r-2}(x,q)-xq^{3p-r}\qmc{3p-r-1}(xq^2,q) + \qmcerr{3p+r}(x,q),
\end{equation*}
where
\begin{equation*}
\qmcerr{3p+r}(x,q) =
\begin{cases}
0 & \text{ if } r=0\\
-x\qmc{3p-1}(x,q)  & \text{ if } r=1\\
-x\qmc{3p-2}(x,q) + q^{p-1}\qmc{3p-4}(x,q) - \delta_{p,1}  & \text{ if } r=2.
\end{cases}
\end{equation*}
Set $x=q^k$ for $k \in \integers$, and let $n \geq 3$; it is clear from these equations that $\qmc{n}(q^k,q)$ lies in the $\zqi$-span of $1$ and the $\qmc{m}(q^p,q)$ for $p \in \bz, \, 0 \leq m <n$.
Since by Theorem \ref{thm: xqarbit}, we have that $\qmc{m}(q^p,q) \in R$ for $p \in \bz, \, 0 \leq m \leq 2$, our proposition now follows by induction.
\end{proof}
\vskip 24 pt
\section{ Demazure modules and the proof of Proposition \ref{secondrecursion} }\label{proofofsecondrecursion}

The goal in  the first part of this section is to collect together the relevant definitions and  results that we shall need to prove Theorem \ref{recursion1} and Proposition \ref{secondrecursion}. We begin this section by briefly reminding the reader of  the definition of  a Demazure module occurring in a highest weight integrable irreducible representation of the  affine Lie algebra $\widehat{\lie{sl}_2}$. We are interested only in stable  Demazure modules and we recall several results from \cite{CV} about this family.  We end the section by proving  Proposition \ref{secondrecursion}.

\subsection{}  Recall that $\lie{sl}_2$ is  the complex simple  Lie algebra of two by two matrices of trace zero and that $\{x,h,y\}$ is  the standard basis of $\lie{sl}_2$, with  $[h,x]=2x$, $[h,y]=-2y$ and $[x,y]=h$. The associated affine Lie algebra  $\widehat{\lie{sl}_2}$ with canonical central element $c$ and scaling operator $d$ can be realized as follows:  as vector spaces we have $$\widehat{\lie{sl}_2}=\lie{sl}_2\otimes\bc[t,t^{-1}]\oplus\bc c\oplus\bc d,$$  where $\bc[t,t^{-1}]$ is the Laurent polynomial ring in an indeterminate $t$, and the  commutator is given by $$[a\otimes f, b\otimes g]=[a,b]\otimes fg,\ \  \ \ [d,a\otimes f]= a\otimes tdf/dt,\ \ \  \ [c,\widehat{\lie{sl}_2}]=0=[d,d].$$  The action of $d$ can also be regarded as defining a $\bz$--grading on $\widehat{\lie{sl}_2}$ where we declare the grade of $d$ and $c$ to be zero and the grade of $a\otimes t^r$ to be $r$ for $a\in\lie{sl}_2$.

 Let $\widehat{\lie h}= \bc h\oplus\bc c\oplus\bc d$ be the Cartan subalgebra and define the Borel and the standard  maximal parabolic subalgebras by  $$\widehat{\lie b}=\lie{sl}_2\otimes t\bc[t]\oplus \bc x\oplus\widehat{\lie h},\qquad \widehat{\lie p}=\widehat{\lie b}\oplus \bc y = \lie{sl}_2\otimes \bc[t]\oplus\bc c\oplus \bc d. $$ Notice that $\widehat{\lie b}$ and $\widehat{\lie p}$ are $\bz_+$--graded subalgebras of $\widehat{\lie g}$. We identify $\lie{sl}_2$ with the grade zero subalgebra $\lie{sl}_2\otimes 1$ of $\lie{sl}_2\otimes \bc[t]$. Define $\delta\in{\widehat{\lie h}}^*$ by: $\delta(d)=1, \delta(\lie h\oplus\bc c)=0$.
Let $\widehat W$ be the affine Weyl group associated to $\widehat{\lie g}$ and recall that  it acts on $\widehat{\lie h}$  and $ \widehat{\lie h}^*$ and leaves $c$ and $\delta$  fixed.
\subsection{} \label{weyldemazure} Suppose that  $\Lambda\in \widehat{\lie h}^*$ is dominant integral: i.e., $\Lambda(h),  \Lambda(c-h)\in\bz_+$ and $\Lambda(d)\in\bz$.   Let $V(\Lambda)$ be the irreducible integrable highest weight  $\widehat{\lie g}$--module generated by a highest weight vector  $v_\Lambda$. The action of $\widehat{\lie h}$ on $V(\Lambda)$ is diagonalizable and  the  central element  $c$ acts via the scalar  $\Lambda(c)$ on $V(\Lambda)$. The non--negative integer  $\Lambda(c)$ is called the level of $V(\Lambda)$.  For all $w\in\widehat{W}$ the element $w\Lambda$ is also an  eigenvalue for the action of $\widehat{\lie h}$  on $V(\Lambda)$ with corresponding eigenspace $V(\Lambda)_{w\Lambda}$. The Demazure module associated to $w$ and $\Lambda$  is defined to be $$V_w(\Lambda)= \bu(\widehat{\lie b}) V(\Lambda)_{w\Lambda}.$$ The Demazure modules are finite--dimensional and
 if  $w\Lambda(h)\le 0$, then $V_w(\Lambda)$ is a module for $\widehat{\lie p}$.   {\em From now on, we shall only be interested in such Demazure modules.}  
 Notice  that these  Demazure modules are indexed by the integers $$-s=w\Lambda(h)\le 0,\ \  \ \  \ell=\Lambda(c),\ \  \ \ p=w\Lambda(d),$$ 
The action of $d$ on the Demazure modules defines a $\bz$--grading on them compatible with $\bz_+$--grading on $\lie{sl}_2[t]$.   Moreover, since $w(\Lambda+p\delta)=w\Lambda+p\delta$ and $(\Lambda+p\delta)(\lie h\oplus\bc c)=\Lambda(\lie h\oplus\bc c)$, it follows that   for a fixed $\ell$ and $s$ the modules are just grade shifts.  If  $s=0$ then $D(\ell, 0)$ is the trivial $\lie{sl}_2[t]$--module.

\subsection{}\label{grch} As the discussion in Section \ref{weyldemazure}  shows, the proper setting for our study is the category of finite--dimensional  $\bz$--graded $\lie{sl}_2[t]$--modules. We recall briefly some of the elementary definitions and properties of this category.   A   finite--dimensional $\bz$--graded $\lie {sl}_2[t]$--module  is a $\bz$--graded vector space  space admitting a compatible graded action of $\lie{sl}_2[t]$:
  $$ V=\bigoplus_{k\in\bz} V[k],\qquad (a\otimes t^r)V[k]\subset V[k+r]\ \  a\in \lie{sl}_2,\ \ r\in\bz_+.$$  In particular, $V[r]$ is a module for the subalgebra $\lie{sl}_2$ of $\lie{sl}_2[t]$ and hence  the action of $\lie h$ on $V[r]$ is semisimple,i.e., $$V[r]=\bigoplus_{m\in\bz} V[r]_m,\ \  V[r]_m=\{v\in V[r]: hv =mv\}.$$ The {\em graded character} of $V$  is the Laurent polynomial in two variables $x,q$ given by $$\ch_{\gr} V=\sum_{m,r\in\bz}\dim V[r]_mx^mq^r.$$

 A map of graded $\lie{sl}_2[t]$--modules is a degree zero map of $\lie{sl}_2[t]$--modules.
 If $V_1$ and $V_2$ are graded $\lie {sl}_2[t]$--modules, then the  direct sum  and tensor product  are again graded  $\lie{sl}_2[t]$--modules, with grading, $$(V_1\oplus V_2)[k]= V_1[k]\oplus V_2[k],\qquad (V_1\otimes V_2)[k]=\bigoplus_{s\in\bz} (V_1[s]\otimes V_2[k-s]).$$  The graded character is additive on short exact sequences and multiplicative on tensor products.

Given a $\bz$--graded vector space $V$, we let $\tau_p^* V$ be the graded vector space whose $r$--th graded piece is $V[r+p]$. Clearly,  a graded action of $\lie{sl}_2[t]$ on $V$ also makes  $\tau_p^* V$ into a graded $\lie{sl}_2[t]$--module.
It is now   easy  to prove (see \cite{CG} for instance)  that  an irreducible object of this category must be  of the form $\tau_p^* V(n)$ where $V(n)$ is the unique (up to isomorphism)  irreducible module for $\lie{sl}_2$ of dimension $(n+1)$.  It follows that if $V$ is an arbitrary finite--dimensional graded $\lie{sl}_2[t]$--module, then $\ch_{\gr} V$ can be written uniquely as a non--negative integer linear combination of $q^p\ch_{\gr}\tau_0^* V(n)$, $p\in\bz$, $n\in\bz_+$.

\subsection{} We recall for the reader's convenience, the graded $\lie{sl}_2[t]$  module $\tau^*_r D(\ell,s)$ defined in Section \ref{mainresult}. Let   $\ell, s\in\bz_+$ and write $s=\ell s_1+s_0$ with $s_1\ge -1$ and $s_0\in\mathbb N$ with $s_0\le \ell$. Then $D(\ell,s)$ is generated by an element $v_s$ and defining relations: \begin{gather}\label{locweyl}(x\otimes \bc[t])v_s=0,\ \ (h\otimes f) v_s= sf(0)v_s,\ \   \ \  (y\otimes 1)^{s+1}v_s=0,\\  \label{demrel2} (y\otimes t^{s_1+1}) v_s=0,\ \ \ \ (y\otimes t^{s_1})^{s_0+1} v_s=0,\ \ \ {\rm {if}}\ \  s_0<\ell.\end{gather} Let $\tau^*_r D(\ell,s)$ be the graded $\lie{sl}_2[t]$-- module obtianed defining  the grade of the element $v_s$ to be $r$. 
The following result is a special case of a result established in \cite[Theorem 2, Proposition 6.7 ]{CV} for $s>0$.
\begin{prop}\label{demdefna} Let $\Lambda$ be a  dominant integral weight for $\widehat{\lie h}$ and let $w\in\widehat W$ be such that $$\Lambda(c)=\ell,\ \ w\Lambda(h)=-s,\ \ w\Lambda(d)=r.$$ We have an isomorphism of  graded $\lie{sl}_2[t]$--modules $$V_w(\Lambda)\cong\tau^*_r D(\ell, s).$$\hfill\qedsymbol\end{prop}

\begin{rem} A few remarks are in order here. In the case when $s_0=\ell$  the second relation in equation \eqref{demrel2} is a consequence of the other relations.  A presentation of all Demazure modules was given in  \cite{Joseph}, \cite{ Ma} in the case of simple and Kac--Moody algebras respectively and includes infinitely many relations of the form $(y\otimes t^a)V_w(\Lambda)=0$.  However, it was shown in \cite[Theorem 2] {CV} that in the case of the $\lie{sl}_2$--stable Demazure modules these relations are all consequences of the oes  stated in the proposition.

\end{rem}

\subsection{} We isolate further results from \cite[Section 6]{CV} that will be needed for our study.  \begin{prop}\label{demdefn} Let   $\ell, s\in\bz_+$ and write $s=\ell s_1+s_0$ with $s_1\ge -1$ and $s_0\in\mathbb N$ with $s_0\le \ell$. \begin{enumerit}

\item[(i)]  For  $0\le s\le \ell$ we have $$D(\ell, s)\cong \tau^*_0 V(s), \ \ {\rm{ i.e.}}\ \ ,\ \ \left(\lie{sl}_2\otimes t\bc[t]\right) D(m,s)=0.$$
\item[(ii)] For $s>0$, we have  $\dim D(\ell,s)=(\ell+1)^{s_1}(s_0+1).$
\item[(iii)] The $\lie{sl}_2[t]$--submodule of $D(\ell,s)$ generated by the element $(y\otimes t^{s_1})^{s_0}v_s$ is isomorphic to $\tau_{s_1s_0}^*D(\ell, s-2s_0)$.  In particular, the quotient  $D(\ell,s)/\tau_{s_1s_0}^*D(\ell, s-2s_0)$ is generated by an element $\bar v_s$ with defining relations, \eqref{locweyl} and, \begin{equation}\label{demrel3} (y\otimes t^{s_1+1})\bar  v_s=0,\ \ \ \ (y\otimes t^{s_1})^{s_0} \bar v_s=0.\end{equation}\hfill\qedsymbol

\end{enumerit}\end{prop}

\subsection{}\label{chdembasis} The following is a straightforward application of the Poincare--Birkhoff--Witt theorem.
\begin{lem} Let $\ell\in\mathbb N$ and $s\in\bz_+$. The module $\tau_0^*V(s)$ is the unique irreducible quotient of $D(\ell,s)$ and occurs with multiplicity one in the Jordan--Holder series of $D(\ell,s)$.  Moreover,  if $\tau_p^* V(m)$, $m\ne s$  is a Jordan--Holder constituent  of $D(\ell,s)$ then $p\in\mathbb N$ and $s-m\in2\mathbb N$.\hfill\qedsymbol
\end{lem}
Let $\ell\in\mathbb N$. It follows from the Lemma that if $V$ is a graded finite--dimensional module for $\lie{sl}_2[t]$, then $\ch_{\gr} V$ can be written uniquely as a  $\bz[q,q^{-1}]$ linear combination of  $\ch_{\gr} D(\ell, s)$, $s\in\bz_+$.

 \subsection{} Let  $V$ be  a finite--dimensional  graded $\lie{sl}_2[t]$--module.
  We say that a decreasing sequence  \\$$ \cal F(V)= \{V=V_0 \supsetneq V_1 \supsetneq \cdots V_k \supsetneq V_{k+1}=0\}$$ of graded $\mathfrak{sl}_2[t]$-submodules of $V$ is  a Demazure flag of level $m$, if \\$$V_i /V_{i+1} \cong \tau_{p_i}^*D(m, n_i),\ \    (n_i, p_i)\in \bz_+\times\bz,\ \ 0 \leq i \leq k.$$  Given a flag $\cal F(V)$ we say that the multiplicity of $\tau_p^* D(m,n)$ in $\cal F(V)$  is the cardinality of the set $\{j: V_j/V_{j+1}\cong \tau_p^* D(m,n)\}$. It is not hard to show that the cardinality of this set is independent of the choice of the Demazure  flag  (see for instance \cite[Lemma 2.1]{CSSW}) of $V$ and we denote this number by $[V:\tau ^*_p D(m,n)]$. Define
 $$[V:D(m, n)]_q=\sum_{p \in\bz}[V:\tau_p^* D(m, n)]q^p,\ \ n\ge 0,\ \ \  [V:D(m,n)]_q=0,\ \ n<0.$$  It follows from the discussion in Section \ref{grch} and Section \ref{chdembasis} that if $V$ admits a Demazure flag of level $m$, then \begin{equation}\label{chgreq}\ch_{\gr} V=\sum_{s\in\bz}[V: D(m,s)]_q\ch_{\gr} D(m,s).\end{equation}
The following result was first proved  in \cite{Naoi} for Demazure modules for arbitrary simply--laced simple algebras using the  theory of canonical basis.
An alternate more constructive  and self conatined proof was given in \cite{CSSW} for $\lie{sl}_2[t]$.
\begin{prop}\label{flagexistence}  Let $\ell$ be a positive integer. For all non--negative integers $s$ and $m$ with $m\geq\ell$, the module $D(\ell, s)$ has a Demazure flag of level $m$. \hfill\qedsymbol\end{prop}
This  proposition along with Lemma \ref{chdembasis} proves that  the initial condition given in  \eqref{initial2} are satisfied.
\subsection{}  Theorem 3.3 of \cite{CSSW} shows  that  there is a very large class of modules admitting a Demazure flag of level $m$.   We do not state that result in full generality since it requires introducing a lot of notation which is not needed in this paper. In the special case we need, Theorem 3.3 and Lemma 3.8 of \cite{CSSW} give the first and second parts of the next proposition.

\begin{prop}\label{cssw} Let $\ell\in\mathbb N$ and $s=\ell s_1+s_0$ with $s_1, s_0\in\bz_+$ and $0<s_0\le \ell$.
\begin{enumerit}
\item[(i)] Consider the embedding $\tau_{s_1s_0}^* D(\ell, s-2s_0)\hookrightarrow D(\ell,s)$. The  corresponding quotient
 admits  a Demazure flag of level $m$ for all $m>\ell$.
\item[(ii)]   We have $$[D(\ell, s)/\tau_{s_1s_0}^*D(\ell,s-2s_0): D(\ell+1, n)]_q= q^{(s-n)/2}[D(\ell, s-\ell-1)): D(\ell+1, n-\ell-1)]_q.$$\end{enumerit}
\hfill\qedsymbol
\end{prop}
The following corollary is immediate.\
\begin{cor} Keep the notation of the proposition. We have
\begin{gather*}[D(\ell, s): D(\ell+1,n)]_q=q^{s_1s_0} [D(\ell, s-2s_0): D(\ell+1,n)]_q +\\ q^{(s-n)/2}[D(\ell, s-\ell-1)):D(\ell+1,n-\ell-1)]_q.\end{gather*}\end{cor}
\subsection{} We can now prove Proposition \ref{secondrecursion}. To prove part (i) of the proposition we proceed by induction on $j$. Since $0\le n\le \ell$ we have by Proposition \ref{demdefn}(i) that $$D(\ell, n)\cong D(\ell+1,n)\cong \tau^*_0 V(n),$$ and so, if  $0\le k\le\ell$,  we get  $[D(\ell,k):D(\ell+1,n)]_q=\delta_{k,n}$.  This shows that  induction begins and for the inductive step we assume that  $$[D(\ell,2j'\ell+k):D(\ell+1,n)]_q=\delta_{k,n}q^{j'(j'\ell +n)},$$ holds for all $0\le j'<j$ and all $0\le k,n\le \ell$. Using Corollary  \ref{cssw} and noting that the second term on the right hand side is zero  since $n\le \ell$,  and using the inductive hypothesis, we get
\begin{align*}
[D(\ell,2j\ell+k)):D(\ell+1,n)]_q&=q^{2kj}[D(\ell,2j\ell-k):D(\ell+1,n)]_q\\
&=q^{2kj}q^{(2j-1)(\ell-k)}[D(\ell,2(j-1)\ell+k):D(\ell+1,n)]_q\\
&=\delta_{k,n}q^{2nj+(2j-1)(\ell-n)+(j-1)(\ell(j-1)+n)} = \delta_{k,n}q^{j(\ell j+n)}.
\end{align*}
This proves the inductive step. It also proves that if $j\ge 1$, then
\begin{align*}
[D(\ell,2j\ell-k):D(\ell+1,n)_q&=q^{(\ell-k)(2j-1)}[D(\ell,2(j-1)\ell+k):D(\ell+1,n)]_q\\ &= \delta_{k,n}q^{(\ell-n)(2j-1)}q^{(j-1)((j-1)\ell+n)}= \delta_{k,n}q^{j(j\ell - n)}.
\end{align*}
This completes the proof of part (i). Part (ii) is precisely the statement of Corollary \ref{cssw}.

\section{Proof of Theorem \ref{recursion1}}\label{proofofrecursion1}  The main idea of the proof is the following. We  study  the tensor product  $D(\ell, s)\otimes D(\ell, 1)$ and write   the graded character of the tensor product explicitly as a linear combination of the graded character of level $\ell$--Demazure modules.   If $m>\ell$, this results allows us to write the graded character of   $D(\ell, s)\otimes D(\ell, 1)$ as a linear combination of the graded character of level $m$ Demazure modules  in two different ways.  
A  comparison of  coefficients then  gives Theorem  \ref{recursion1}.

\subsection{}  The  proof of the following Proposition can be found in  Section \ref{beginproof}- \ref{proofend}.
\begin{prop}\label{tensorproduct} Let $\ell$ be a positive integer and let  $s\in\bz_+$. Write $s=\ell s_1+s_0$ with $s_1, s_0\in\bz$, $s_1\ge -1$  and $0<s_0\le \ell$. We have,  \begin{gather*}{\ch}_{\gr} D(\ell,s )\ch_{\gr} D(\ell, 1)=\ch_{\gr}D(\ell ,s+1)+(1-\delta_{s_0,\ell})\ch_{\gr} D(\ell,s-1)\\ + \ q^{s_1(s_0-\ell\delta_{s_0,\ell})}(1-q^{s_1+\delta_{s_0,\ell}})\ch_{\gr}D(\ell, s-2(s_0-\ell\delta_{s_0,\ell})-1).\end{gather*}

\end{prop}

\begin{rem}\label{rem:notation_dict}
Let $s$ be as in the proposition.
If we let $r(s,\ell)$ be the unique integer with $0 \leq r(s,\ell) <\ell$ such that
$s = \ell \,\floor{\frac{s}{\ell}} + r(s,\ell)$, we have
$$ \delta_{s_0,\ell} = \delta_{r(s,\ell),\, 0}\;\;, \;\;\;\; r(s,\ell) = s_0-\ell\delta_{s_0,\ell}\;\;, \;\;\;\; \floor{\frac{s}{\ell}} = s_1+\delta_{s_0,\ell}.$$
In particular, this means $r(s,\ell)\,\delta_{s_0,\ell} = 0$ and hence $  r(s,\ell)\,\floor{\frac{s}{\ell}} = r(s,\ell)\,s_1$. Using these relations, Proposition \ref{tensorproduct} can be reformulated in terms of
$\floor{\frac{s}{\ell}}$ and $r(s,\ell)$ in place of $s_1, s_0$.
\end{rem}

\subsection{} We now prove Theorem \ref{recursion1}. We first explain the strategy of the proof.
Using  equation \eqref{chgreq} and Proposition \ref{flagexistence}, we can write,
\[
\ch_{\gr}D(\ell,s)=\sum_{p\ge 0}[D(\ell,s):D(m,p)]_q \ch_{\gr}D(m,p),
\] where $m\in\bz_+$ with $m \ge \ell$.
Multiplying both sides of the equation by $\ch_{\gr}D(\ell, 1)$ gives,
\begin{equation}\label{equality}
\ch_{\gr}D(\ell,s)\ch_{\gr}D(\ell,1)=\sum_{n\ge 0}[D(\ell,s):D(m,p)]_q \ch_{\gr}D(m,p)\ch_{\gr}D(m,1).
\end{equation}
Here, we have used the fact that $D(\ell, 1)\cong D(m, 1)$ (see Proposition \ref{demdefn}(i)) as $\lie{sl}_2[t]$-modules.
Now, recall that the product of graded characters is the graded character of the tensor product. We can therefore apply Proposition \ref{tensorproduct}  to both sides of the preceding equation . Applying it to the right hand side gives us a linear combination of the graded characters of level $m$--Demazure modules. Applying it to the left hand side, gives a linear combination of graded characters of level $\ell$--Demazure modules. These can be further expressed as a combination of the graded characters of  level $m$--Demazure modules.  Equating  the coefficients of a level $m$--Demazure module on both    sides will prove Theorem \ref{recursion1}.

\medskip

In this subsection, it will be more convenient to work with the notation suggested by Remark \ref{rem:notation_dict}.
Let  us collect the coefficients of $\ch_{\gr}D(m,n)$ which occur on the right hand side of equation \eqref{equality} after applying Proposition \ref{tensorproduct}. It  can  occur with non--zero coefficients only  in the products:  $\ch_{\gr} D(m,n \pm 1)\ch_{\gr} D(m,1)$  and in $\ch_{\gr}D(m,p)\ch_{\gr} D(m,1)$, where $$p-2r(p,m)-1=n.$$  We claim that this implies  \begin{equation} \label{p}
p=2m + n-2r(n,m)-1.
\end{equation}
To prove this, we consider $x = p+n+1.$ Since $ x= 2 \left( p - r(p,m)\right)$, it is clearly a multiple of $2m$.
Further, since $ p = n+1 + 2r(p,m)$, we have $$ n+1 \leq p \leq n+1+2(m-1).$$
This implies $$ 2n+2 \leq x \leq 2n+2m.$$
Thus, we deduce that $x$ is the unique multiple of $2m$ that lies within these bounds; it is given by
$$ x = 2m \left(\floor{\frac{2n+2m}{2m}}\right) = 2m\,\left( \floor{\frac{n}{m}} + 1 \right),$$
or equivalently by
$$ x = 2m+2n - r(2m+2n,2m) = 2m+2n - 2r(n,m).$$
Thus, $p = x-n-1$ is given by the required expression.

Summarizing (and using Remark \ref{rem:notation_dict} again),
we find that the coefficient of $\ch_{\gr}D(m,n)$ on the right hand side is:
\begin{gather}\label{eq:rhs}
[D(\ell,s):D(m,n-1)]_q+(1-\delta_{r(n+1,m),\,0})[D(\ell,s):D(m,n+1)]_q\\+
q^{r(p,m)\,\floor{\frac{p}{m}}}(1- q^{\floor{\frac{p}{m}}}) [D(\ell,s): D(m,p)]_q,\notag
\end{gather} where $p$ is as in \eqref{p}. We note from \eqref{p} that
\beq \label{eq:psimp}
r(p,m) = m-r(n,m) - 1 \text{ and } \floor{\frac{p}{m}} = \frac{p - r(p,m)}{m} = 1 + \floor{\frac{n}{m}}.
\eeq
Now, we apply Proposition \ref{tensorproduct} to the left hand side of equation \eqref{equality}. This gives us a linear combination of graded characters of level $\ell$-Demazure modules which we can then rewrite using \eqref{chgreq}. We find then that the resulting coefficient of $\ch_{\gr}D(m,n)$ is:
\begin{gather} \label{eq:lhs}
[D(\ell,s+1):D(m,n)]_q+(1-\delta_{r(s,\ell),\,0})[D(\ell,s-1):D(m,n)]_q\\ +
q^{r(s,\ell)\,\floor{\frac{s}{\ell}}}(1- q^{\floor{\frac{s}{\ell}}})[D(\ell,s-2r(s,\ell)-1):D(m,n)]_q.\notag
\end{gather}

Setting \eqref{eq:rhs} and \eqref{eq:lhs} equal to each other and using \eqref{eq:psimp},
 we obtain Theorem \ref{recursion1}. \qed

\subsection{} \label{beginproof} The rest of the section is devoted to the proof of Proposition \ref{tensorproduct}.  If $s=0$, then $D(\ell, 0)$ is the trivial module and the propostion is trivially true. So, from now on we assume that $s>0$.  For the proof we consider three mutually exclusive  cases and it is helpful to write down the equality of characters according to these cases:
\begin{enumerit}
\item[(i)]  If $0<s=s_0<\ell$, then \begin{equation}\label{s=s_0}\ch_{\gr} D(\ell,s)\otimes D(\ell,1)=\ch_{\gr} D(\ell, s+1)+\ch_{\gr} D(\ell, s-1).\end{equation}\\
\item[(ii)] If $s_0=\ell$ (in particular if $\ell=1$),  then
 \begin{equation}\label{s_0=l}\ch_{\gr}( D(\ell,s)\otimes D(\ell,1))=\ch_{\gr} D(\ell, s+1)+(1-q^{s_1+1})\ch_{\gr}D(\ell, s-1).\end{equation}\\
\item[(iii)] If $s>\ell> s_0$,  then  \begin{gather}\label{case3}{\ch}_{\gr}( D(\ell,s )\otimes  D(\ell, 1))=\ch_{\gr}D(\ell ,s+1)+\ch_{\gr} D(\ell,s-1)+ \ q^{s_1s_0}(1-q^{s_1})\ch_{\gr}D(\ell, s-2s_0-1).\end{gather}\\
\end{enumerit}

\subsection{}  \label{defineu0u2} By  Proposition \ref{demdefn}(i) we know   that $D(\ell, 1)\cong\tau_0^*V(1)$ for all $\ell \in\bz_+$. In particular, the elements $v_1, yv_1$  are a basis of $D(\ell, 1)$ where we have identified the element $y\in\lie{sl}_2$ with $y\otimes 1$ in $\lie{sl}_2[t]$.
From now on for ease of notation, we set $$U_0=D(\ell ,s)\otimes D(\ell, 1).$$
\begin{lem}  We have $U_0\cong \bu(\lie{sl}_2[t])(v_s\otimes yv_1)$.
\end{lem}
\begin{pf} Since $y^2 v_1=0$ we have  $$(y\otimes t^k)(v_s\otimes yv_1)= (y\otimes t^k)v_s\otimes yv_1,\ \ k\ge 0.$$ Repeating this argument we get that the $\lie{sl}_2[t]$--submodule generated by $v_s\otimes yv_1$ contains the subspace $D(\ell,s)\otimes yv_1$. Since $x(D(\ell,s)\otimes yv_1)=D(\ell, s)\otimes v_1 + (xD(\ell,s))\otimes yv_1,$ the Lemma is established.
\end{pf}
Set $$U_2=\bu(\lie{sl}_2[t])(v_s\otimes v_1).$$
 It is trivial to check that for all $f\in\bc[t]$, we have  \begin{gather}\label{vsv1} \ \
 \  (x\otimes f)(v_s\otimes v_1)= 0,\ \ (h\otimes f)(v_s\otimes v_1)=f(0)(s+1)(v_s\otimes v_1),\ \ (y\otimes 1)^{s+1}(v_s\otimes v_1)=0,\end{gather} and also that \begin{gather}\label{vsyv1}(x\otimes f)(v_s\otimes yv_1)\in U_2,\ \ \ (h\otimes f)=f(0)(s-1)(v_s\otimes yv_1), \ \ (y\otimes 1)^{s}(v_s\otimes yv_1)\in U_2.\end{gather}

\subsection{} We now prove that  equation \eqref{s=s_0} is satisfied. Since $s=s_0<\ell$, we see by using Proposition \ref{demdefn}(i) that  $$(\lie{sl}_2\otimes t\bc[t])(v_s\otimes v_1)=0,\ \qquad  U_2\cong \tau^*_0V(s+1)\cong D(\ell, s+1).$$ Since the graded character is additive on short exact sequences, it suffices now  to prove that $U_0/U_2\cong D(\ell, s-1)$.   Equation \eqref{vsyv1} and the fact that $(\lie{sl}_2\otimes t\bc[t])(v_s\otimes yv_1)=0$ shows that  that the image of $v_s\otimes yv_1$ in $U_0/U_2$ satsifies the relations of $D(\ell, s-1)$ given in Proposition \ref{demdefna}.  Since  $D(\ell,s-1)\cong \tau^*_0 V(s-1)$ is irreducible we see that $U_0/U_2\cong  D(\ell,s-1)$  and \eqref{s=s_0} follows.

\subsection{} To prove  the remaining two cases, we need the following result established in   \cite[Lemma 2.3, Equation (2.10)]{CV}. For any $m\in\bz_+$ and  $a\in \bu(\lie{sl}_2[t])$ let $a^{(m)}=a^m/m!$.
Given a positive integer $r$ and a non--negative integer $p$, define elements $\boy(r,p)\in\bu(\lie{sl}_2[t])$ by

 \[
{\bf y}(r,p)=\sum (y\otimes 1)^{(b_0)}\cdots (y\otimes t^p)^{(b_p)}
\]
where  the sum is over all $p$--tuples $(b_0,\cdots, b_p)$ such that $r=\sum_j b_j,\, p=\sum_j jb_j$.
 \begin{prop}\label{altpres}  Let $\ell$ be a positive integer and $s=\ell s_1+s_0$  with $s_1,s_0\in\bz_+$ and $0<s_0\le \ell$. Then   $D(\ell, s)$ is the $\lie{sl}_2[t]$--module generated by an element $v_s$ with the  relations given in \eqref{locweyl} and the relation  \[{\bf y}(r,p)v_s=0
\]
for all $r, p\in\bz_+$ satisfying,  $p\ge r s_1+1$ or   $r+p\ge 1 +rk+\ell(s_1-k)+s_0$ for some $0\le k\le s_1$.\hfill\qedsymbol
\end{prop}

\subsection{} We now consider the case when $s_0=\ell$, i.e., $s=\ell(s_1+1)$. We shall prove that there exists surjective maps of graded $\lie{sl}_2[t]$--modules $$\varphi_1: D(\ell, s+1)/\tau^*_{s_1+1} D(\ell, s-1)\to U_2\to 0,\ \ \varphi_2: D(\ell, s-1)\to U_0/U_2\to 0.$$ Once this is done, the proof of \eqref{s_0=l} is completed as follows. By Proposition \ref{demdefn}(ii), we have  $$\dim D(\ell, s+1)=2(\ell+1)^{s_1+1} = \dim U_0=\dim U_0/U_2+\dim U_2,
$$ and hence $\varphi_1$ and $\varphi_2$ must be isomorphisms. Using the additivity of $\ch_{\gr}$  gives \eqref{s_0=l}.

To prove the existence of $\varphi_1$,   use  Proposition \ref{demdefna} and Proposition \ref{demdefn}(iii) with $s$ replaced by $s+1=\ell(s_1+1)+1$. In view of \eqref{vsv1} it suffices to   prove that $ (y\otimes t^{s_1+1})(v_s\otimes v_1)=0$.  But this is obvious since $(y\otimes t^{s_1+1})v_s= 0 =(y\otimes t^{s_1+1})v_1$.

To prove the existence of $\varphi_2$, note that $s-1=\ell s_1+\ell-1$. In view of \eqref{vsyv1} we see that   we   only  have to prove that  $$(y\otimes t^{s_1+1})(v_s\otimes yv_1)\in U_2,\ \ \ell>1,\ \ \ (y\otimes t^{s_1})^{\ell}(v_s\otimes yv_1)\in U_2,\ \ \ell\ge 1.$$ The idea in both  cases is the same: namely for all $p\ge 0$ and $r\ge 1$, we can write $$(y\otimes t^p)^r(v_s\otimes yv_1)= (y\otimes t^p)^ry(v_s\otimes v_1)- C((y\otimes t^p)^ryv_s)\otimes v_1, $$  for some $C\in\bc$. Since the first term on the right hand side is in $U_2$ the left  hand side  will be in $U_2$ iff the second term on the right hand side is also in $U_2$. In other words, we must prove that
 \begin{equation}\label{one} ((y\otimes t^{s_1+1})yv_s)\otimes v_1)\in U_2,\ \ \ell>1,\ \ ((y\otimes t^{s_1})^{\ell}yv)_s\otimes v_1)\in U_2,\ \ \ell\ge 1.\end{equation}
If $\ell>1$. then $ ((y\otimes t^{s_1+1})yv_s)\otimes v_1)=0$ since $(y\otimes t^{s_1+1})v_s=0$ and the first assertion of \eqref{one} is established.  To prove the second assertion suppose first that $s_1=0$, i.e., $s=\ell$ .Then equation \eqref{locweyl} gives $  (y\otimes 1)^{\ell}yv_\ell=y^{\ell+1}v_\ell=0$ and we are done. If $s_1>0$,
 take  $r=\ell+1$, $p=\ell s_1$  and $k=0$ in Proposition \ref{altpres} and observe  that $$\boy(\ell+1,  \ell s_1)v_s=0.$$  Suppose that $b_0,\cdots , b_{\ell s_1}$ are such that $\sum_{j=0}^{\ell s_1}b_j=\ell+1$ and $\sum_{j=1}^{\ell s_1} jb_j=\ell s_1$. If $b_m> 0$ for any $m\ge s_1+1$ then  $(y\otimes t^m)v_s=0$ and so $$ (y\otimes 1)^{(b_0)}\cdots (y\otimes t^{\ell s_1})^{(b_{\ell s_1})}v_s=0.$$  Suppose now that  $b_j=0$ for all $j>s_1$ and $b_0>1$. Then,   we have $$\sum_{j=1}^{s_1}b_j< \ell ,  \ \  \ell s_1= \sum_{j=1}^{s_1}jb_j \le s_1\sum_{j=1}^{s_1}b_j< \ell s_1,$$ which is absurd.  Hence $b_0\le 1$. If $b_0=1$ and $b_m>0$ for $0<m<s_1$,  then we again have $$\ell s_1=\sum_{j=1}^{s_1}jb_j\le  s_1\left(\sum_{j\ne m}b_j\right)+ mb_m<s_1\sum_{j=1}^{s_1}b_j=\ell s_1,$$ which is again absurd. Hence we find  that $$0= \boy(\ell+1,\ell s_1)v_s= (y\otimes 1)(y\otimes t^{s_1})^{\ell} v_s  + X v_s$$ where $X\in\bu(\lie {sl_2}\otimes t\bc[t])$ is an element of grade $\ell s_1>0$. This gives,
$$((y\otimes 1)(y\otimes t^{s_1})^{\ell} v_s)\otimes v_1= - Xv_s\otimes v_1= -X(v_s\otimes v_1)\in U_2$$ and the proof of \eqref{one} is complete.

\subsection{}\label{proofend} For the final case of $s>\ell>s_0$, we
need an additional submodule,  \begin{gather*} U_1= U_2+   \bu(\lie{sl_2}[t])(y\otimes t^{s_1})^{s_0}(v_s\otimes yv_1)= U_2+\bu(\lie{sl_2}[t])((y\otimes t^{s_1})^{s_0}v_s)\otimes yv_1.\end{gather*}  We will show the existence of three surjective morphisms of graded $\lie{sl}_2[t]$--modules: \begin{gather*}\psi_1:  D(\ell, s+1)/\tau^*_{s_1(s_0+1)} D(\ell,s-2s_0-1)\to U_2\to 0,\\ \psi_2:   \tau^*_{s_1s_0} D(\ell, s-2s_0-1)\to U_1/U_2\to 0,\ \ \psi_3:  D(\ell,s-1) \to U_0/U_1\to 0.
\end{gather*}
The proof is then completed as in the preceding case: a dimension count shows that the maps $\psi_j$, $j=1,2,3$ must be isomorphisms and the equality of graded characters follows.  The proof of the existence of the maps is also very similar to the proofs given for $\varphi_j$, $j=1,2$, and  we provide the details only in the case of the module $U_1/U_2$  which is more complicated.  Thus, for $\psi_2$ to exist
 we  must prove that    \begin{gather}
\label{u11} (x\otimes \bc[t])((y\otimes t^{s_1})^{s_0}v_s)\otimes yv_1\in U_2,  \ \  ((h\otimes t\bc[t])(y\otimes t^{s_1})^{s_0}v_s)\otimes yv_1) =0 ,
\end{gather}
as well as: if $s_0<\ell-1$, \begin{gather}\label{dem1} (y\otimes t^{s_1})(y\otimes t^{s_1})^{s_0}(v_s\otimes yv_1)\in U_2,\ \  (y\otimes t^{s_1-1})^{\ell-s_0}(y\otimes t^{s_1})^{s_0}(v_s\otimes yv_1)\in U_2\end{gather}  and if $s_0=\ell-1$,\begin{equation}\label{dem2} (y\otimes t^{s_1-1})(y\otimes t^{s_1})^{s_0}(v_s\otimes yv_1)\in U_2.\end{equation}

 For \eqref{u11}, it is enough to  note that $xyv_1= v_1$ and that  Proposition \ref{demdefn}(iii) implies that  $$ (x\otimes \bc[t])(y\otimes t^{s_1})^{s_0}v_s=\ 0= \  (h\otimes t\bc[t]))v_s.$$
 Since $s_1\ge 1$ we have,  $$(y\otimes t^{s_1})(y\otimes t^{s_1})^{s_0}(v_s\otimes yv_1)=(y\otimes t^{s_1})^{s_0+1}v_s\otimes yv_1 =0,$$ where the last equality is from \eqref{demrel2}.  This proves the first assertion in  \eqref{dem1}.

To prove the second assertion in \eqref{dem1} and \eqref{dem2}, 
we argue as in the proof of the existence of  map $\varphi_2$ that  $$(y\otimes t^{s_1-1})^{\ell-s_0}(y\otimes t^{s_1})^{s_0}(v_s\otimes yv_1)\in U_2\iff ( (y\otimes t^{s_1-1})^{\ell-s_0}(y\otimes t^{s_1})^{s_0}yv_s)\otimes v_1\in U_2.$$
 Taking $r=\ell+1$, $p=s-\ell$ and $ k=0$ we  see by using Proposition \ref{altpres} that $$\boy(\ell+1, s-\ell)v_s=0.$$ Suppose that $((y\otimes 1)^{(b_0)}\cdots (y\otimes t^{s-\ell})^{(b_{s-\ell})})$,  is an expression occurring in $\boy(\ell+1, s-\ell)$. Then its action on $v_s$ is zero if $b_j>0$ for some $j\ge s_1+1$.  Moreover, by Proposition  \ref{demdefn}(iii), we have $$(y\otimes t^{s_1})^{s_0+1}v_s=0,\ \ (y\otimes t^{s_1-1})^{\ell-s_0+1}(y\otimes t^{s_1})^{s_0}v_s=0,$$ it follows that  we may assume also that \begin{equation}\label{bs1} b_{s_1}\le s_0,\ \ b_{s_1-1}\le \ell-s_0.\end{equation} If $s_1=1$, this forces $b_1=s_0$ and $b_0=\ell+1-s_0$ and hence we have proved that $$0=\boy(\ell+1, s-\ell)= y^{\ell+1-s_0}(y\otimes t)^{s_0}v_s\in U_2$$
Suppose that  $s_1>1$ and $b_0>0$. Then $\sum_{j=1}^{s_1} b_j\le \ell$ and we get $$s-\ell=\sum_{j=1}^{s_1}jb_j\le \left((s_1-2)\sum_{j=1}^{s_1}b_j \right)+ b_{s_1-1}+2b_{s_1}\le  \ell(s_1-2) +b_{s_1-1}+2b_{s_1},$$ i.e $b_{s_1-1}+2b_{s_1}\ge \ell +s_0$. Using equation \eqref{bs1}, we see that we must have $b_{s_1-1}=\ell-s_0$ and $b_{s_1}=s_0$ and hence $b_0=1$ and $b_m=0$ if $m\notin\{0,s_1-1, s_1\}$.

This proves that the element, $$0=\boy(\ell+1, s-\ell)v_s= ((y\otimes t^{s_1-1})^{\ell-s_0}(y\otimes t^{s_1})^{s_0}y) v_s  +Xv_s$$ where $X\in\bu(\lie{sl}_2\otimes t\bc[t])$. Since $Xv_s\otimes v_1= X(v_s\otimes v_1)$ it follows that $$ ((y\otimes t^{s_1-1})^{\ell-s_0}(y\otimes t^{s_1})^{s_0}y) v_s \otimes v_1\in U_2.$$

\end{document}